\documentclass[11pt,reqno]{amsart}
\usepackage[usenames,dvipsnames,svgnames,table]{xcolor}
\usepackage{amssymb,amsfonts,amsmath,amscd,mathrsfs}
\usepackage{mathtools,verbatim,enumitem,url}
\usepackage[all]{xy}

\usepackage{tikz-cd}
\usepackage{hyperref}

\newtheorem{theorem}{Theorem}[section]
\newtheorem{lemma}[theorem]{Lemma}     
\newtheorem{corollary}[theorem]{Corollary}
\newtheorem{proposition}[theorem]{Proposition}

\theoremstyle{definition}
\newtheorem{definition}[theorem]{Definition}
\newtheorem{remark}[theorem]{Remark}

\newtheorem{notation}[theorem]{Notation} 
\newtheorem{discussion}[theorem]{Discussion}
\newtheorem{convention}[theorem]{Convention}

\newcommand{\mbk}{\mbox{$\Bbbk$}}
\newcommand{\mbl}{\mbox{$\mathbb{L}$}}
\newcommand{\mbn}{\mbox{$\mathbb{N}$}}
\newcommand{\mbq}{\mbox{$\mathbb{Q}$}}
\newcommand{\mbz}{\mbox{$\mathbb{Z}$}}

\newcommand{\mon}{\mbox{${\rm x}$}}
\newcommand{\monx}{\mbox{$\underline{\rm x}$}}
\newcommand{\monset}{\mbox{${\rm Mon}$}}

\newcommand{\mfa}{\mbox{$\mathfrak{a}$}}
\newcommand{\mfm}{\mbox{$\mathfrak{m}$}}

\newcommand{\msb}{\mbox{$\mathscr{B}$}}
\newcommand{\msc}{\mbox{$\mathscr{C}$}}
\newcommand{\msd}{\mbox{$\mathscr{D}$}}

\newcommand{\fac}{\mbox{${\rm F}$}}
\newcommand{\frob}{\mbox{${\rm Frob}$}}

\newcommand{\lser}{\mbox{$[\mkern-1.5mu [ $}}
\newcommand{\rser}{\mbox{$]\mkern-1.5mu ] $}}
\newcommand{\llaur}{\mbox{$(\!( $}}
\newcommand{\rlaur}{\mbox{$)\!) $}}
\newcommand{\ord}{\mbox{${\rm ord}$}}
\newcommand{\ds}{\mbox{${\rm ds}$}}
\newcommand{\sord}{\mbox{${\rm ord}_{\sigma}$}}
\newcommand{\lform}{\mbox{${\rm LF}$}}

\newcommand{\lterm}{\mbox{${\rm LT}$}}
\newcommand{\lmon}{\mbox{${\rm LM}$}}
\newcommand{\supp}{\mbox{${\rm Supp}$}}

\newcommand{\nums}{\mbox{$\mathcal{S}$}}
\newcommand{\charac}{\mbox{${\rm char}$}}

\newcommand{\im}{\mbox{${\rm Im}$}}
\newcommand{\height}{\mbox{${\rm ht}$}}
\newcommand{\grade}{\mbox{${\rm grade}$}}
\newcommand{\depth}{\mbox{${\rm depth}$}}
\newcommand{\length}{\mbox{${\rm length}$}}
\newcommand{\rank}{\mbox{${\rm rank}$}}
\newcommand{\card}{\mbox{${\rm card}$}}
\newcommand{\nf}{\mbox{${\rm NF}$}}
\newcommand{\nfmora}{\mbox{${\rm NFMora}$}}

\title[Prime ideals of Moh and the characteristic of the field]
{Prime ideals of Moh and the characteristic of the field}

\author[L. Gonz\'alez]{Laura Gonz\'alez} 
\author[F. Planas-Vilanova]{Francesc Planas-Vilanova}

\address{Departament de Matem\`atiques, Universitat Polit\`ecnica de
Catalunya. Diagonal 647, ETSEIB, E-08028 Barcelona.} 
\email{laura.gonzalez.hernandez@upc.edu}
\email{francesc.planas@upc.edu} 
\thanks{Both authors are partially supported by Grant PID2023-146936NB-I00 financed by the Spanish State Agency MCIN/AEI/10.13039/501100011033/ FEDER, UE and AGAUR project 2021 SGR 00603}

\date{\today}



\begin{document}

\begin{abstract}
We reprove and generalize a result of Moh which gives a lower bound on the minimal number of generators of an ideal in a power series ring in three variables $x,y,z$ over a field $\mbk$. As a consequence, in each characteristic of the field $\mbk$, we obtain a minimal generating set for the prime ideal $P_3$ of Moh. We deduce that the minimal number of generators of $P_3$ might decrease depending on the characteristic of $\mbk$. This contradicts a statement of Sally and leaves as an open problem to find families of prime ideals in $\mbk\lser x,y,z\rser$ with an unbounded minimal number of generators, when $\mbk$ has characteristic other than zero. Finally, we show that these minimal generating sets of $P_3$ are standard basis with the negative degree reverse lexicographic order.
\end{abstract}

\maketitle

\section{Introduction}

Let $\mbk$ be a field, $n\geq 1$ be an odd integer, $m=(n+1)/2$ and $\lambda$ an integer greater than $n(n+1)m$, with $\gcd(\lambda,m)=1$. In his papers \cite{moh1,moh2}, Moh considers the following family 
of $\mbk$-algebra morphisms $\rho_n:\mbk\lser x,y,z\rser\to\mbk\lser t\rser$, where $\rho_n(x)=t^{nm}+t^{nm+\lambda}$, $\rho_n(y)=t^{(n+1)m}$ and $\rho_n(z)=t^{(n+2)m}$. The {\em prime ideal $P_n$ of Moh} is defined as $P_n:=\ker(\rho_n)$. Let
$\mu(I)$ stand for the minimal number of generators of an ideal $I$. 
Moh proves that, if $\mbk$ has characteristic zero, then $\mu(P_n)=n+1$. Shortly after,  
Sally gives in her book \cite{sally} an excellent overview of Moh's work. In a Remark in pages 61-62, she writes
``Once Moh's primes are constructed in $\mbk\lser x,y,z\rser$, with $\mbk$ of characteristic 0, then one can show that they also exist in $\tilde{\mbk}\lser x,y,z\rser$, for $\tilde{\mbk}$ a field of any characteristic.'' In her discussion, she affirms that, if $\overline{P}_n$ is the corresponding ideal in characteristic $p>0$, then $\mu(\overline{P}_n)\geq \mu(P_n)$. Much more recently, Mehta, Saha and Sengupta in \cite{mss} revisit Moh's primes, always in characteristic zero, focusing on giving an iterative procedure for finding a set of polynomials 
that minimally generates $P_n$. Note that in his papers Moh does not affirm, at least explicitly,  that the ideals $P_n$ can be generated by polynomials. 

In the present paper we first study, what is to us, the key result of Moh in \cite{moh1}: a lower bound on the minimal number of generators of an ideal $Q$ of $\mbk\lser x,y,z\rser$ in terms of the dimension of some $\mbk$-vector spaces $V_r=V_r(\sigma,Q)$ associated to an order $\sigma$ in $\mbk\lser x,y,z\rser$ and to the ideal $Q$ 
(see \cite[Theorem~4.3]{moh1}). It is readily seen that that theorem holds true in any dimension and for any characteristic of $\mbk$. Our main contribution here is twofold. In
Theorem~\ref{moh43-1}, given a generating set $\msc$ of $Q$, we show 
a method of reducing $\msc$ to a new generating set of $Q$, 
contained in the $\mbk$-vector space spanned by $\msc$, and
satisfying certain conditions related to $V_r(\sigma,Q)$, from where one deduces the lower bound given by Moh. 
In Theorem~\ref{moh43-2}, we extend sets satisfying certain conditions related to $V_r=V_r(\sigma,Q)$ to minimal generating sets of $Q$.

By means of this second result, we construct, for each characteristic of the field $\mbk$, a set of polynomials that minimally generates the prime ideal $P_3$ of Moh ($n=3$, so $m=2$), for $\lambda=25$. We show that $\mu(P_3)=4$, not only in characteristic $0$, but also in characteristic $p\geq 5$ (see Theorem~\ref{theo:char0}). 
Surprisingly enough, we prove that $\mu(P_3)=2$ in characteristic $2$ and $\mu(P_3)=3$ in characteristic $3$ (see Theorems~\ref{theo:char2} and~\ref{theo:char3}). This
contradicts the aforementioned inequality $\mu(\overline{P}_n)\geq \mu(P_n)$ in the book of Sally (see Remark~\ref{remark-sally2}). And what is more, it leaves as an open problem to find a subfamily of the family of prime ideals of Moh, or an alternative family of prime ideals in $\mbk\lser x,y,z\rser$, with an unbounded minimal number of generators, and where $\mbk$ is a field of characteristic other than 0. We hope to come back to this question in a near future. 

Finally, we show that the minimal generating sets provided for $P_3$ are, in fact, standard basis with the negative degree reverse lexicographic order $>_{\rm ds}$ on the set of monomials 
(see the definition in Convention~\ref{grevlex}).


\section{Notations and preliminaries}\label{section-notations}

This section is devoted to introduce some terminology and some easy preliminary results.

\begin{notation}\label{not-S}
We denote the set of non-negative integers by $\mbn$. Fix $d\in\mbn$, $d\geq 3$.
Let $n_1,\ldots,n_d$ in $\mbn$, $n_i\geq 1$, with $\gcd(n_1,\ldots,n_d)=1$. 
Let $\nums=\langle n_1,\ldots,n_d\rangle$ be the numerical
semigroup generated by $n_1,\ldots,n_d$. For
$r\in\mbn$, let
\begin{eqnarray*}
\fac(r,\nums)=\{\alpha=(\alpha_1,\ldots,\alpha_d)\in\mbn^d\mid
\alpha_1 n_1+\cdots+\alpha_dn_d=r\}
\end{eqnarray*}
stand for the {\em set of factorizations of} $r$. Note that if
$r\not\in\nums$, then $\fac(r,\nums)=\emptyset$. 
\end{notation}

\begin{notation}\label{not-sigma}
Let $\mbk$ be a field, $\monx=x_1,\ldots,x_d$ a set of variables over $\mbk$ and  
$\monset_d=\{\mon^\alpha\mid \alpha\in\mbn^d\}$ the set of monomials $\mon^{\alpha}=x_1^{\alpha_1}\cdots
x_d^{\alpha_d}$ with multi-index exponent $\alpha=(\alpha_1,\ldots,\alpha_d)\in\mbn^d$. The degree of $\mon^{\alpha}$
is defined as $\deg(\mon^\alpha)=|\alpha|=\alpha_1+\cdots+\alpha_d$. 
Let $\mbk\lser\monx\rser=\mbk\lser x_1,\ldots,x_d\rser$ denote the power series ring with coefficients in $\mbk$. 
Any power series can be written as $f=\sum_{\alpha\in\mathbb{N}^d}\lambda_{\alpha}\mon^{\alpha}$, with
$\lambda_{\alpha}\in\mbk$. The {\em support of $f$}, $f\neq 0$, is $\supp(f)=\{\alpha\in\mbn^d\mid\lambda_{\alpha}\neq 0\}$.  The {\em order of $f$}, $f\neq 0$, is 
$\ord(f)=\min\{|\alpha|\in\mbn\mid \alpha\in\supp(f)\}$. The {\em leading form of $f$}, $f\neq 0$, is $\lform(f)=\sum_{|\alpha|={\rm ord}(f)}\lambda_{\alpha}\mon^{\alpha}$; $f$ is said to be {\em homogeneous} if $f=\lform(f)$. If $f,g$ are homogeneous, both of order $r$, then $f+g=0$ or $f+g$ is homogeneous of order $r$.
In this paper, whenever we work with power series rings in three variables or in one variable, 
they will be denoted $x,y,z$ or $t$, respectively. For instance, if $d=3$ and $f=z^3-2xy^3+x^2yz-y^5z^4\in\mbk\lser x,y,z\rser$, then $\supp(f)=\{(0,0,3),(1,3,0),(2,1,1),(0,5,4)\}$, $\ord(f)=3$ and
$\lform(f)=z^3$.

Let $\sigma:\mbk\lser\monx\rser\rightarrow\mbk\lser\monx\rser$ be the $\mbk$-algebra morphism defined by $\sigma(x_i)=x_i^{n_i}$, for
$i=1,\ldots,d$. Given $f=\sum_{\alpha\in\mathbb{N}^d}\lambda_{\alpha}\mon^\alpha
\in \mbk\lser\monx\rser$, $f\neq 0$, the $\sigma$-{\em order} of $f$ is defined as 
$\sord(f):=\ord(\sigma(f))$. Note that $\sord(f)\in\nums$. 
The $\sigma$-{\em leading form} of $f$ is
$f^{\sigma}:=\sigma^{-1}(\lform(\sigma(f)))=
\sum_{\alpha\in{\rm F}(r,\mathcal{S})}\lambda_{\alpha}\mon^{\alpha}$, where $r=\sord(f)$; $f$ is
said to be $\sigma$-{\em homogeneous} if $\sigma(f)$ is homogeneous, that is, if
$f=f^{\sigma}$. Let $f^{\tau}:=f-f^{\sigma}$, which will be called the {\em tail} of
$f$. Following with the previous example, if $\nums=\langle 3,4,5\rangle$, and $\sigma(x)=x^3$, $\sigma(y)=y^4$ and $\sigma(z)=z^5$, then $\sigma(f)=
z^{15}-2x^3y^{12}+x^6y^4z^5-y^{20}z^{20}$, 
$\sord(f)=15\in\nums$, $\lform(\sigma(f))=z^{15}-2x^3y^{12}+x^6y^4z^5$,
$f^\sigma=z^3-2xy^3+x^2yz$ and $f^{\tau}=-y^5z^4$.

Given a subset $\msb$ of $\mbk\lser\monx\rser$, $\msb\neq\emptyset,\{0\}$, let
$\msb^{\sigma}=\{f^{\sigma}\mid f\in\msb, f\neq 0\}$ denote the set of its
$\sigma$-leading forms. If $\msb=\emptyset$ or $\msb=\{0\}$, then we just define
$\msb^{\sigma}=\emptyset$.
\end{notation}    

\begin{remark}\label{remark-Wr}
Let $r\in\mbn$. Set $W_r=\{g\in\mbk\lser\monx\rser\mid g\neq 0,\; g=g^\sigma ,\; \sord(g)=r\}\cup \{0\}$. 
Then, $g\in W_r$ if and only if $g=\sum_{\alpha\in{\rm F}(r,\mathcal{S})}\lambda_{\alpha}\mon^{\alpha}$.
Since $\fac(r,\nums)$ is finite, $W_r=\langle \mon^{\alpha}\mid
\alpha\in\fac(r,\nums)\rangle$ is the $\mbk$-vector space spanned by the 
monomials $\mon^{\alpha}$, with $\alpha\in\fac(r,\nums)$. Thus, the elements of $W_r$ are polynomials. 
If $r\in\nums$, then $W_r\neq \{0\}$ and $\dim W_r\geq 1$.  
If $r=0$, then $W_0=\langle 1\rangle_{\Bbbk}$, the $\mbk$-vector space spanned by $1$.  If $r\not\in\nums$, we just define $W_r=\{0\}$. Observe that $W_r=W_r(\sigma)$ depends on the chosen 
$\sigma$.

Given $h=\sum_{\alpha\in\mathbb{N}^d}\lambda_{\alpha}\mon^\alpha\in
\mbk\lser\monx\rser$ and $r\in\nums$, let $h_{(r)}=\sum_{\alpha\in{\rm F}(r,\mathcal{S})}\lambda_{\alpha}\mon^{\alpha}\in W_r$ be the     
summation of all the monomial terms of $h$ of $\sigma$-order $r$. Since $\fac(r,\nums)$ is finite, then
$h_{(r)}$ is a polynomial and $\sum_{r\in\mathcal{S}}h_{(r)}$ 
is well-defined and coincides with $h$. 
\end{remark}

\begin{lemma}\label{lemma-Vr}
Let $Q$ be a proper ideal of $\mbk\lser\monx\rser$, i.e., $Q\neq\mbk\lser\monx\rser$. 
Let $r\in\nums$. Set 
\begin{eqnarray*}
V_r=\{g\in W_r\mid g\neq 0,\mbox{ there exists }f\in Q, f\neq 0, \mbox{ such that
}g=f^{\sigma}\}\cup\{0\}.
\end{eqnarray*}
Then $V_r$ is a $\mbk$-vector subspace of $W_r$. Moreover, if $g\in W_r$, $g\neq 0$, then
the following are equivalent:
\begin{itemize}
\item[$(a)$] $g\in V_r$;
\item[$(b)$] $g\in Q$ or there exists $h\in\mbk\lser\monx\rser$,
$h\neq 0$, such that $\sord(h)>r$ and such that $f:=g+h\in Q$. 
\end{itemize}
In such a case, $g=f^{\sigma}$ and $h=f^{\tau}$. 
\end{lemma}
\begin{proof}
If $g_1,g_2\in V_r$, with $g_1=-g_2$, then $g_1+g_2=0\in V_r$.
Take $g_1, g_2\in V_r$, with $g_1,g_2\neq 0$ and $g_1\neq -g_2$. Then
$g_1,g_2\in W_r$ and $g_1+g_2\in W_r$, because $W_r$ is a vector space. 
By definition, there are $f_i\in Q$, $f_i\neq 0$, such that $g_i=f_i^{\sigma}$, for $i=1,2$. Then, $f_1+f_2\in Q$ 
and $f_1+f_2\neq 0$, otherwise $g_1=f_1^{\sigma}=-f_2^{\sigma}=-g_2$. In particular, 
$f_1+f_2=f_1^\sigma+f_1^\tau +f_2^\sigma +f_2^\tau=g_1+g_2+f_1^\tau+f_2^\tau$, where $g_1+g_2\in W_r$, $g_1+g_2\neq 0$, and $f_1^\tau+f_2^\tau=0$ or
$\sord(f_1^\tau+f_2^\tau)>r$. It follows that $(f_1+f_2)^\sigma=g_1+g_2$. Therefore $g_1+g_2\in V_r$.
Similarly, if $g\in V_r$, $g\neq 0$, and $\lambda\in\mbk$, $\lambda\neq 0$, 
then $\lambda g\in W_r$, because $W_r$ is a vector space. 
By definition, there exists $f\in Q$, $f\neq 0$, such that $g=f^\sigma$. Then, 
$\lambda f\in Q$, $\lambda f\neq 0$. 
In particular, $\lambda f=\lambda f^\sigma+\lambda f^\tau=\lambda g+\lambda f^{\tau}$, 
where $\lambda g\in W_r$, $\lambda g\neq 0$, and $\lambda f^{\tau}=0$ or 
$\sord(\lambda f^{\tau})>r$. It follows that 
$(\lambda f)^{\sigma}=\lambda g$. Therefore, $\lambda g\in V_r$. 
Thus, $V_r$ is a $\mbk$-vector space. 

Suppose that $g\in W_r$, $g\neq 0$, and that $g\in V_r$. By definition of $V_r$, there exists 
$f\in Q$, $f\neq 0$, and such that $g=f^{\sigma}$. In particular, $\sord(f)=r$.
If $f^{\tau}=0$, then $g=f^{\sigma}=f$, so $g\in Q$. If $f^{\tau}\neq 0$, take $h=f^{\tau}$. 
Then $h\neq 0$, $\sord(h)>r$, and $f=f^{\sigma}+f^{\tau}=g+h$. Conversely, 
let $g\in W_r$, $g\neq 0$, 
and such that $g$ satisfies the condition $(b)$. 
In particular, $g=g^{\sigma}$ and $\sord(g)=r$. If $g\in Q$, take $f=g$, 
so $f\in Q$, $f\neq 0$, $g=g^{\sigma}=f^{\sigma}$ and $g\in V_r$. 
If there exists $h\in\mbk\lser\monx\rser$, $h\neq 0$, such that $\sord(h)>r$ and such that $f=g+h\in Q$, 
then, clearly, $f^{\sigma}=g$ and $g\in V_r$.
\end{proof}

Note that, since $Q$ is proper, then $Q$ has no invertible elements,
there does not exist $f\in Q$, $f\neq 0$, with $\ord(f)=0$ and
$V_0=\{0\}$. If $r\not\in\nums$, we just define $V_r=\{0\}$. 
If $Q=0$, then $V_r=\{0\}$, for every $r\in\mbn$.
Observe that $V_r=V_r(\sigma,Q)$ depends on the choice of $\sigma$ and $Q$. 

\section{Revisiting a theorem of Moh}\label{sectionmoh43}\label{section-revisiting}

A careful reading of the proof of Theorem~4.3 of Moh in \cite{moh1} reveals a 
very deep and rich result (see also the proof of the main result, just quoted as ``Theorem'', in \cite{moh2}). 
In this section we bring to the surface the somehow hidden ideas in the proof of Moh, which lead to Theorems~\ref{moh43-1} and \ref{moh43-2}. 

Recall that, by convention, or just by definition, the empty set is the unique basis of the vector space whose only vector is zero. Let $\card(\msc)$ stand for the cardinality of a set $\msc$ and $\langle \msc\rangle_{\Bbbk}$ for the $\mbk$-vector space spanned by $\msc$. Keeping the terminology as in Notation~\ref{not-S} 
and \ref{not-sigma}, we have:

\begin{theorem}\label{moh43-1} 
Let $Q$ be a proper ideal of $\mbk\lser\monx\rser$ and let
$\msc$ be a finite system of generators of $Q$.
Let $\sigma:\mbk\lser\monx\rser\rightarrow\mbk\lser\monx\rser$ be 
the $\mbk$-algebra morphism defined by $\sigma(x_i)=x_i^{n_i}$, for $i=1,\ldots,d$,
and $V_r=V_r(\sigma,Q)$. Let $\xi=\min(n_1,\ldots,n_d)$,
$s=\min\{\sord(f)\mid f\in Q, f\neq 0\}$, with $\xi, s\geq 1$.
Then, there exist $\msb_0,\ldots,\msb_{\xi-1},\msc_{\xi-1}$ such that, 
for each $i=0,\ldots,\xi-1$, 
\begin{itemize}
\item[$\bullet$] $\msb_0,\ldots,\msb_{\xi-1},\msc_{\xi-1}$ are included in $\langle \msc\rangle_{\Bbbk}$.
\item[$\bullet$] $\msb_i^{\sigma}$ is a (possibly empty) basis of $V_{s+i}$
and, if $g\in\msc_{\xi-1}$, $g\neq 0$, then $\sord(g)\geq s+\xi$.
\item[$\bullet$] $\msb_0,\ldots,\msb_{\xi-1},\msc_{\xi-1}$ are pairwise disjoint.
\item[$\bullet$] $\sum_{i=0}^{\xi-1}\dim V_{s+i}\leq \card(\msb_0\cup\ldots\cup\msb_{\xi-1}\cup\msc_{\xi-1})\leq\card(\msc)$.
\item[$\bullet$] $\msb_0\cup\ldots\cup\msb_{\xi-1}\cup\msc_{\xi-1}$ 
is a generating set of $Q$. 
\end{itemize}
In particular, if $\msc$ is a minimal generating set of $Q$, so it is 
$\msb_0\cup\ldots\cup\msb_{\xi-1}\cup\msc_{\xi-1}$, and
\begin{eqnarray*}
\mu(Q) \geq \sum_{i=0}^{\xi -1} \dim V_{s+i}\, .
\end{eqnarray*}
\end{theorem}

Let us comment some aspects of Theorem~\ref{moh43-1}, especially by comparison 
to Moh's Theorem~4.3 in \cite{moh1}. 
First observe that it holds in any dimension $d\geq 3$ and in any characteristic of the field $\mbk$. 
While Moh focused on the lower bound of $\mu(Q)$, we are interested in 
detailing the construction of a generating set of $Q$ 
from where one can deduce that lower bound. 
The proof consists on substituting the elements of a given system of generators $\msc$ of $Q$ 
by other elements, included in $\langle \msc\rangle_{\Bbbk}$, 
and whose leading forms determine bases of the subspaces $V_{s+i}$. 
In particular, if the elements in $\msc$ are polynomials, 
then so they are the elements in $\langle\msc\rangle_{\Bbbk}$. 
Therefore, in such a case, the elements in $\msb_i$ and $\msc_{\xi-1}$
are polynomials as well. 

So, Theorem~\ref{moh43-1} gives a theoretical method of reducing a generating set $\msc$ 
of the ideal $Q$ to a new one satisfying a series of conditions. 
However, one can not expect the new generating set to be simply $\msc$, or a subset of $\msc$. 
Instead, the elements of the new generating set are $\mbk$-linear combinations of the elements of $\msc$. For instance, let $\sigma(x)=x^3$, $\sigma(y)=y^4$, 
$\sigma(z)=z^5$, $f_1=x^3$, $f_2=x^3+y^4$, $f_3=z^5$, 
$\msc=\{f_1,f_2,f_3\}$ and $Q=(f_1,f_2,f_3)=(x^3,y^4,z^5)$. Then, 
$V_3=W_3=\langle x^3\rangle$, $V_4=W_4=\langle y^4\rangle$ and $V_5=W_5=\langle z^5\rangle$ 
(see Remark~\ref{remark-Wr} and Lemma~\ref{lemma-Vr}). Moreover, $\xi=3$ and $s=3$. 
Thus, one must take $\msb_0=\{x^3\}$, $\msb_1=\{y^4\}$ and $\msb_2=\{z^5\}$, 
because each $\msb_i^{\sigma}$ is the unique basis of $V_{s+i}$, up to a multiplicative non-zero constant. 
However, $\msb_1$ is not contained in $\msc$, but it is in $\langle\msc\rangle_{\Bbbk}$. 

We have not been able to see in the papers of Moh a proof that the ideals $P_n$ are generated by 
polynomials. This is not the case in the paper of Mehta, Saha and Sengupta, in characteristic zero, where they give an iterative method to construct minimal generating sets of $P_n$ whose elements are polynomials (see \cite[Section~4]{mss}). 

Using Theorem~\ref{moh43-1}, and with the terminology as in Notation~\ref{not-S} 
and \ref{not-sigma}, we can deduce:

\begin{theorem}\label{moh43-2}
Let $Q$ be a proper ideal of $\mbk\lser\monx\rser$ and let
$\msc$ be a minimal generating set of $Q$.
Let $\sigma:\mbk\lser\monx\rser\rightarrow\mbk\lser\monx\rser$ be 
the $\mbk$-algebra morphism defined by $\sigma(x_i)=x_i^{n_i}$, for $i=1,\ldots,d$, and
$V_r=V_r(\sigma,Q)$. Let $\xi=\min(n_1,\ldots,n_d)$,
$s=\min\{\sord(f)\mid f\in Q, f\neq 0\}$, with $\xi, s\geq 1$.
Let $\msd_0,\ldots,\msd_{\xi-1}\subset Q$ be such that,
for each $i=0,\ldots,\xi-1$, $\msd_i^{\sigma}$ is a (possibly empty)
linearly independent subset of $V_{s+i}$. 
Then, $\msd_0\cup\ldots\cup\msd_{\xi-1}$ can be extended
to a minimal generating set of $Q$ with elements of $\langle \msc\rangle_{\Bbbk}$.
\end{theorem}

This second result asserts that it is possible to extend
$\msd_0\cup\ldots\cup\msd_{\xi-1}$, where $\msd_i^{\sigma}$ is a linearly independent subset of $V_{s+i}$, to a minimal generating set of $Q$. Again, if the elements of $\msc$ are polynomials, then 
$\msd_0\cup\ldots\cup\msd_{\xi-1}$ can be extended to a minimal generating set of $Q$ by adding polynomials. We will use this second theorem to find, in each characteristic of $\mbk$,  a minimal generating set of the prime ideal $P_3$ of Moh. 

In the end, two are the main ingredients in the proofs of Theorems~\ref{moh43-1} and \ref{moh43-2}: 
the equality $\sord(fg)=\sord(f)+\sord(g)$, where $f,g\in\mbk\lser\monx\rser$, $f,g\neq 0$, 
in the first result, and the Steinitz Exchange Lemma, a ubiquitous procedure in linear algebra, 
in the second result. Summarizing, Theorems~\ref{moh43-1} and \ref{moh43-2} can be thought as two theorems 
on ``reducing to" and ``extending to" generating sets, respectively.

\begin{proof}[Proof of Theorem~\ref{moh43-1}]
Take $\msc=\{f_1,\ldots,f_{\eta}\}$, a finite set
of generators of $Q$ of cardinality $\eta\geq 1$. 
For $r\geq 1$, set $d_r=\dim V_r$. We shall iteratively construct, in each step 
$0\leq i\leq \xi-1$, pairs of sets $\msb_i$, $\msc_i$, so that, in the end, 
$\msb_0,\ldots,\msb_{\xi-1},\msc_{\xi-1}$ satisfy the five required conditions. 

\vspace*{0,2cm}

\noindent \underline {\sc Step $0$: Construction of $\msb_0$ and $\msc_0$}. 

\vspace*{0,2cm}

\noindent By definition of $s$, $V_s\neq \{0\}$. Take
$\{v_{0,1},\ldots,v_{0,d_s}\}\subset Q$ such that
$\{v^\sigma_{0,1},\ldots,v^{\sigma}_{0,d_s}\}$ is a basis of $V_s$.
For each $1\leq j\leq d_{s}$, since $v_{0,j}\in Q$ and $Q=(f_1,\ldots,f_{\eta})$, then
\begin{eqnarray*}
v_{0,j}=\sum_{k=1}^{\eta} a_{0,j}^kf_k,
\end{eqnarray*}
where $a_{0,j}^k\in\mbk\lser\monx\rser$. Note that, if $a_{0,j}^k$ is
not a unit and $a_{0,j}^k\neq 0$, then 
$\sord(a_{0,j}^kf_k)=\sord(a_{0,j}^k)+\sord(f_k)\geq \xi+s>s=\sord(v_{0,j})$.  In
particular, the subset $\Lambda_{0,j}$ of indices $l$ of
$\{1,\ldots,\eta\}$ such that $a_{0,j}^l$ is invertible, with term of
order zero $\lambda_{0,j}^l\in\mbk$, $\lambda_{0,j}^l\neq 0$, and such
that $\sord(f_l)=s$, is not empty. Taking $\sigma$-leading forms in
both parts of the equality, we obtain
\begin{eqnarray*}
v_{0,j}^{\sigma}=\sum_{l\in\Lambda_{0,j}}\lambda_{0,j}^lf_l^{\sigma}.
\end{eqnarray*}
Let $\Lambda_0=\Lambda_{0,1}\cup\ldots\cup\Lambda_{0,d_s}$. Then
\begin{eqnarray*}
V_s=\langle v^\sigma_{0,1},\ldots,v^{\sigma}_{0,d_s}\rangle\subseteq
\langle f_l^{\sigma}\mid l\in\Lambda_0\rangle\subseteq V_s.
\end{eqnarray*}
Thus $V_s=\langle f_l^{\sigma}\mid l\in\Lambda_0\rangle$. Since $V_s$
has dimension $d_s$ and $\{f_l^{\sigma}\mid l\in\Lambda_0\}$ spans
$V_s$, then $\card(\Lambda_0)\geq d_s$, and there exists a subset
$\Gamma\subseteq \Lambda_0\subseteq\{1,\ldots,\eta\}$, $\Gamma$ of
cardinality $d_s$, such that $\{f_l^{\sigma}\mid l\in\Gamma\}$ is a
basis of $V_s$. In particular $d_s\leq\eta$. Renaming the elements
$f_1,\ldots,f_{\eta}$, we can suppose that $\{f_1^{\sigma},\ldots,
f_{d_s}^{\sigma}\}$ is a basis of $V_s$. For $l=1,\ldots,d_s$, set
$f_{0,l}:=f_l$ and let $\msb_0=\{f_{0,1},\ldots,f_{0,d_s}\}$. Note
that $\msb_0\subseteq\msc$, $d_s=\card(\msb_0)\leq \card(\msc)=\eta$ and
$\msb_0^\sigma$ is a basis of $V_s$.

\vspace*{0,2cm}

Once $\msb_0$ has been constructed, let us build $\msc_0$.

\vspace*{0,2cm}

\noindent \underline{Suppose that $\msb_0=\msc$}. 
Let us prove that $V_r=\{0\}$, for all $s<r<s+\xi$. 
Since $\msb_0=\msc$, then $\msb_0^{\sigma}=\msc^{\sigma}=
\{f_1^{\sigma},\ldots ,
f_{\eta}^{\sigma}\}$ is a basis of $V_s$.  If $\xi=1$, there is 
nothing to prove. Suppose that $\xi\geq 2$, and that there exists
$g\in V_r$, $g\neq 0$, for some $s<r<s+\xi$. Then $g=f^\sigma$, with
$f\in Q$, $\sord(f)=r$. Since $f\in Q=(f_1,\ldots,f_\eta)$, then 
$f=\sum_{j=1}^\eta b_jf_j$, where
$b_j\in\mbk\lser\monx\rser$. Let $\Lambda\subseteq \{1,\ldots,\eta\}$
denote the subset of indices $l$ such that $b_{l}$ is invertible, with
term of order zero $\mu_{l}\in\mbk$, $\mu_{l}\neq 0$. Suppose that
$\Lambda\neq\emptyset$. Since $\sord(f)=r>s$, then, necessarily,
$\sum_{l\in\Lambda}\mu_lf_l^{\sigma}=0$. Since $\{f_1^{\sigma},\ldots
, f_{\eta}^{\sigma}\}$ is a basis of $V_s$, this implies $\mu_l=0$, a
contradiction. Therefore, $\Lambda=\emptyset$ and $b_j$ are non-units, for all $j=1,\ldots,\eta$. 
But then $\sord(b_jf_j)\geq \xi+s>r$, whereas $\sord(f)=r$, a
contradiction again. Therefore, $f$ must be zero and $V_{s+1}=\ldots=V_{s+\xi-1}=\{0\}$.

So, when $\msb_0=\msc$, take $\msb_{1}=\ldots=\msb_{\xi-1}=\emptyset$ and 
$\msc_0=\msc_1=\ldots=\msc_{\xi-1}=\emptyset$. Then, 
$\msb_0,\ldots,\msb_{\xi-1},\msc_{\xi-1}$ 
satisfy the five required properties, namely, 
$\msb_0,\ldots,\msb_{\xi-1},\msc_{\xi-1}$ are included in $\langle \msc\rangle_{\Bbbk}$,
$\msb_i^{\sigma}$ is a (possibly empty) basis of $V_{s+i}$,
$\msb_0,\ldots,\msb_{\xi-1},\msc_{\xi-1}$ are pairwise disjoint sets,
$\sum_{i=0}^{\xi-1}\dim V_{s+i}=\card(\msb_0\cup\ldots\cup\msb_{\xi-1}\cup\msc_{\xi-1})=\card(\msc)$
and $\msb_0\cup\ldots\cup\msb_{\xi-1}\cup\msc_{\xi-1}$ is a generating set of $Q$. 

\vspace*{0,2cm} 

\noindent \underline{Suppose that $\msb_0\subsetneq\msc$}.
For any $g\in\msc\setminus\msb_0$, with
$\sord(g)=s$, if any, then
$g^{\sigma}=\sum_{l=1}^{d_s}\lambda_lf_{0,l}^{\sigma}$, for some
$\lambda_l\in\mbk$. Moreover, $g-\sum_{l=1}^{d_s}\lambda_lf_{0,l}\in \langle\msc\rangle_{\Bbbk}$
and $\sord(g-\sum_{l=1}^{d_s}\lambda_lf_{0,l})>s$. Let
\begin{multline*}
\msc_{0,1}:=\left\{g\in\msc\setminus\msb_0\mid\sord(g)>s\right\}\subseteq\msc\setminus\msb_0
\mbox{ and }\\
\msc_{0,2}:=\left\{g-\sum_{l=1}^{d_s}\lambda_lf_{0,l}\bigg|\;
g\in\msc\setminus\msb_0\;, \sord(g)=s\;, 
g^{\sigma}=\sum_{l=1}^{d_s}\lambda_lf_{0,l}^{\sigma}\right\}\subset \langle\msc\rangle_{\Bbbk}.
\end{multline*}
Take $\msc_0=\msc_{0,1}\cup\msc_{0,2}$. Since $\msb_0\subsetneq\msc$, then $\msc_0\neq\emptyset$. Indeed, 
there exists $g\in\msc\setminus\msb_0\subset Q$, so $\sord(g)\geq s$; if $\sord(g)>s$, then $g\in\msc_{0,1}\subseteq\msc_0$; if $\sord(g)=s$, then 
$g-\sum_{l=1}^{d_s}\lambda_lf_{0,l}\in\msc_{0,2}\subseteq\msc_0$. 

Since the elements of $\msc_0$ have $\sigma$-order bigger than $s$, then
$\msb_0$ and $\msc_0$ are disjoint sets, so 
$\card(\msb_0\cup\msc_0)=\card(\msb_0)+\card(\msc_0)$. On the other hand, 
$\msc_{0,1}$ and $\msc_{0,2}$ are not necessarily disjoint sets, so
\begin{multline*}
\card(\msc_0)=\card(\msc_{0,1}\cup\msc_{0,2})\leq\card(\msc_{0,1})+\card(\msc_{0,2})\leq \\
\card(\{g\in\msc\setminus\msb_0\mid\sord(g)>s\})+
\card(\{g\in\msc\setminus\msb_0\mid\sord(g)=s\})=
\card(\msc\setminus\msb_0).
\end{multline*}
Therefore,
\begin{eqnarray*}
d_s=\card(\msb_0)<\card(\msb_0)+\card(\msc_0)=\card(\msb_0\cup\msc_0)\leq\card(\msc).
\end{eqnarray*}

Since, $\msb_0,\msc_{0,1},\msc_{0,2}\subset \langle\msc\rangle_{\Bbbk}$, then the ideal generated by
$\msb_0\cup\msc_0$ is contained in the ideal generated by $\msc$, which is $Q$. Thus, 
$(\msb_0\cup\msc_0)\subseteq (\msc)=Q$. Let us prove that 
$Q=(\msc)\subseteq (\msb_0\cup\msc_0)$. Take $g\in\msc$. If $g\in\msb_0$ or 
$g\in\msc\setminus\msb_0$ and $\sord(g)>s$, then, clearly, 
$g\in\msb_0\cup\msc_{0,1}\subset\msb_0\cup\msc_0$.
If $g\in\msc\setminus\msb_0$ and $\sord(g)=s$, 
then $g=\sum_{l=1}^{d_s}\lambda_lf_{0,l}+(g-\sum_{l=1}^{d_s}\lambda_lf_{0,l})$ and $g\in\langle\msb_0\cup\msc_{0,2}\rangle_{\Bbbk}\subset (\msb_0\cup\msc_0)$.
Therefore, $\msc\subset (\msb_0\cup\msc_0)$, 
$Q=(\msc)\subseteq (\msb_0\cup\msc_0)$ and 
$\msb_0\cup\msc_0$ is a generating set of $Q$. 

So, when $\msb_0\subsetneq\msc$, then we obtain two sets 
$\msb_0=\{f_{0,1},\ldots,f_{0,d_s}\}$ and $\msc_0=\{g_{0,1},\ldots,g_{0,e_s}\}$, say, with
$\card(\msb_0)=d_s>0$ and $\card(\msc_0)=e_s>0$, such that: $\msb_0$, $\msc_0$ are
included in $\langle\msc\rangle_{\Bbbk}$; $\msb_0^{\sigma}$ is a basis of $V_s$ and 
$\sord(g_{0,l})>s$; 
$\msb_0$ and $\msc_0$ are disjoint; $d_s<\card(\msb_0\cup\msc_0)\leq\card(\msc)$
and $\msb_0\cup\msc_0$ is a generating set of $Q$.

\vspace*{0,2cm} 

If $\xi=1$, we are done. Suppose that $1\leq \xi-1$, so in particular, $\xi\geq 2$.

\vspace*{0,2cm} 

\noindent \underline{\sc Step $1$: Construction of $\msb_1$ and $\msc_1$}. 

\vspace*{0,2cm}

\noindent By {\sc Step 0}, we can suppose that 
$\msb_0\subsetneq\msc$ and $\msc_0\neq\emptyset$, otherwise we take $\msb_{1}=\ldots=\msb_{\xi-1}=\emptyset$ 
and $\msc_0=\msc_1=\ldots=\msc_{\xi-1}=\emptyset$. 

\vspace*{0,2cm}

\noindent \underline{Suppose that $V_{s+1}=\{0\}$}. 
Then, take $\msb_1=\emptyset$ and $\msc_1=\msc_{0}$, so
$\msb_0\cup\msb_{1}\cup\msc_{1}=\msb_0\cup\msc_{0}$, 
which is a set of generators of $Q$ and 
$d_s+d_{s+1}=d_s\leq \card(\msb_0\cup\msb_1\cup\msc_1)\leq \card(\msc)$.

\vspace*{0,2cm}

\noindent \underline{Suppose that $V_{s+1}\neq \{0\}$}.
Take $\{v_{1,1},\ldots,v_{1,d_{s+1}}\}\subset Q$ such that
$\{v^\sigma_{1,1},\ldots,v^{\sigma}_{1,d_{s+1}}\}$ is a basis of
$V_{s+1}$.  Since $v_{1,j}\in Q$ and $Q$ is generated by the elements
of $\msb_0\cup\msc_{0}$, then, for each $1\leq j\leq d_{s+1}$,
\begin{eqnarray}\label{eq-cas1}
v_{1,j}=\sum_{l=1}^{d_{s}}a^{0,l}_{1,j}f_{0,l}
+\sum_{l=1}^{e_s}b^{l}_{1,j}g_{0,l},
\end{eqnarray}
where $a^{0,l}_{1,j},b^l_{1,j}\in\mbk\lser\monx\rser$,
$f_{0,l}\in\msb_0$ and $g_{0,l}\in\msc_{0}$. Note that
$\sord(v_{1,j})=s+1$ and $\sord(f_{0,l})=s$.
In addition, $\sord(g_{0,l})\geq s+1$.
If $a^{0,l}_{1,j}$ is not a unit and $a^{0,l}_{1,j}\neq 0$, then
$\sord(a^{0,l}_{1,j}f_{0,l})\geq\xi+s>s+1=\sord(v_{1,j})$, because $\xi\geq 2$. Let
$\Lambda^0_{1,j}\subseteq \{1,\ldots,d_{s}\}$ denote the subset of
indices $l$ such that $a^{0,l}_{1,j}$ is invertible, with term of
order zero $\lambda^{0,l}_{1,j}\in\mbk$, $\lambda^{0,l}_{1,j}\neq 0$,
and suppose that $\Lambda^0_{1,j}\neq\emptyset$. Taking
$\sigma$-leading forms in Equality~\eqref{eq-cas1}, we deduce that
$\sum_{l\in\Lambda^0_{1,j}}\lambda^{0,l}_{1,j}f_{0,l}^{\sigma}=0$. Since
$\{f^{\sigma}_{0,1},\ldots,f^{\sigma}_{0,d_{s}}\}$ is a basis of
$V_s$, then $\lambda^{0,l}_{1,j}=0$, a contradiction. Therefore, all
$a^{0,l}_{1,j}$ are non-units. Hence,
$\sord(a^{0,l}_{1,j}f_{0,l})>s+1=\sord(v_{1,j})$, for all
$l=1,\ldots,d_{s}$. Let
$\Lambda_{1,j}\subseteq \{1,\ldots,e_s\}$ be the set of
indices $l$ such that $b^l_{1,j}$ is invertible, with term of order
zero $\mu^{l}_{1,j}\in\mbk$, $\mu^l_{1,j}\neq 0$, and such that
$\sord(g_{0,l})=s+1$. Note that $\Lambda_{1,j}\neq\emptyset$ for, if
$b_{1,j}^l$ is not a unit, then $\sord(b_{1,j}^lg_{0,l})\geq
\xi+s+1> \sord(v_{1,j})$. Taking $\sigma$-leading forms in
Equality~\eqref{eq-cas1}, we deduce that
\begin{eqnarray*}
v_{1,j}^{\sigma}=\sum_{l\in\Lambda_{1,j}}\mu^l_{1,j}g_{0,l}^{\sigma}.
\end{eqnarray*}
Let $\Lambda_1=\Lambda_{1,1}\cup\ldots\cup\Lambda_{1,d_{s+1}}\subseteq\{1,\ldots,e_s\}$. Then
\begin{eqnarray*}
V_{s+1}=\langle
v_{1,1}^{\sigma},\ldots,v_{1,d_{s+1}}^{\sigma}\rangle\subseteq \langle
g_{0,l}^{\sigma}\mid l\in\Lambda_1\rangle\subseteq V_{s+1}.
\end{eqnarray*}
So, $V_{s+1}=\langle g_{0,l}^{\sigma}\mid l\in\Lambda_1\rangle$.
Since $\dim V_{s+1}=d_{s+1}$ and $\{g_{0,l}^{\sigma}\mid l\in\Lambda_1\}$
spans $V_{s+1}$, then $\card(\Lambda_1)\geq d_{s+1}$, and there exists
a subset $\Gamma\subseteq\Lambda_1\subseteq\{1,\ldots,e_s\}$,
$\card(\Gamma)=d_{s+1}$, such that $\{g_{0,l}^{\sigma}\mid
l\in\Gamma\}$ is a basis of $V_{s+1}$. In particular,
$d_{s+1}=\card(\Gamma)\leq \card(\Lambda_1)\leq e_s$.
Renaming the elements of
$\msc_{0}=\{g_{0,1},\ldots,g_{0,e_s}\}$, we can suppose
that $\{g^{\sigma}_{0,1},\ldots,g^{\sigma}_{0,d_{s+1}}\}$ is a
basis of $V_{s+1}$. For $l=1,\ldots,d_{s+1}$, set $f_{1,l}:=g_{0,l}$
and let $\msb_1=\{f_{1,1},\ldots,f_{1,d_{s+1}}\}$. Note that
$\msb_1\subseteq\msc_{0}\subset \langle\msc\rangle_{\Bbbk}$ and $\card(\msb_1)=d_{s+1}$.

\vspace*{0,2cm}

Once $\msb_1$ has been constructed, let us build $\msc_1$.

\vspace*{0,2cm}

\noindent \underline{Suppose that $\msb_1=\msc_0$}. 
Then, $\msb_0\cup\msb_1=\msb_0\cup\msc_{0}$, which is a system of
generators of $Q$ by the former step.
Let us show that $V_r=\{0\}$, for all
$s+1<r<s+\xi$. If $\xi=2$, there is nothing to prove. Suppose that $\xi\geq 3$,
and that there exists $g\in V_r$, $g\neq 0$, for some
$s+1<r<s+\xi$. Then $g=f^\sigma$, with $f\in Q$, $\sord(f)=r$. Since
$f\in Q=(\msb_0\cup\msb_1)$, then
\begin{eqnarray*}
f=\sum_{l=1}^{d_{s}}a_{0,l}f_{0,l}+\sum_{l=1}^{d_{s+1}}a_{1,l}f_{1,l},
\end{eqnarray*}
where $a_{k,l}\in\mbk\lser\monx\rser$. Let $\Lambda_0\subseteq
\{1,\ldots,d_s\}$ denote the subset of indices $l$ such that $a_{0,l}$
is invertible, with term of order zero $\lambda_{0,l}\in\mbk$,
$\lambda_{0,l}\neq 0$. Suppose that $\Lambda_0\neq\emptyset$. Since
$\sord(f_{0,l})=s<s+1<r=\sord(f)$, then, necessarily,
$\sum_{l\in\Lambda_0}\lambda_{0,l}f_{0,l}^{\sigma}=0$. Since
$\{f_{0,1}^{\sigma},\ldots,f_{0,d_{s}}^{\sigma}\}$ is a basis of
$V_s$, this implies $\lambda_{0,l}=0$, a contradiction. Therefore,
$\Lambda_0=\emptyset$, all $a_{0,l}$ are non-units and
$\sord(a_{0,l}f_{0,l})\geq \xi+s\geq s+3$. 
Let $\Lambda_1\subseteq
\{1,\ldots,d_{s+1}\}$ denote the subset of indices $l$ such that $a_{1,l}$
is invertible, with term of order zero $\lambda_{1,l}\in\mbk$,
$\lambda_{1,l}\neq 0$. Suppose that $\Lambda_1\neq\emptyset$. Since
$\sord(f_{1,l})=s+1<r=\sord(f)$ and
$\sord(f_{1,l})=s+1<s+3\leq s+\xi\leq\sord(a_{0,l'}f_{0,l'})$, then, necessarily,
$\sum_{l\in\Lambda_1}\lambda_{1,l}f_{1,l}^{\sigma}=0$. Since
$\{f_{1,1}^{\sigma},\ldots,f_{1,d_{s+1}}^{\sigma}\}$ is a basis of
$V_{s+1}$, this implies $\lambda_{1,l}=0$, a contradiction. Therefore
$\Lambda_1=\emptyset$ and all $a_{1,l}$ are non-units. But, then,
$\sord(a_{1,l}f_{1,l}),\sord(a_{0,l'}f_{0,l'})\geq \xi+s>r$, whereas $\sord(f)=r$, a
contradiction again. Therefore, $f$ must be zero and 
$V_{s+2}=\ldots=V_{s+\xi-1}=\{0\}$.

So, when $\msb_1=\msc_0$, take $\msb_{2}=\ldots=\msb_{\xi-1}=\emptyset$ and 
$\msc_1=\ldots=\msc_{\xi-1}=\emptyset$. Then, $\msb_0,\ldots,\msb_{\xi-1},\msc_{\xi-1}$ 
satisfy the required properties:
$\msb_0,\ldots,\msb_{\xi-1},\msc_{\xi-1}\subset\langle \msc\rangle_{\Bbbk}$;
$\msb_i^{\sigma}$ is a (possibly empty) basis of $V_{s+i}$;
$\msb_0,\ldots,\msb_{\xi-1},\msc_{\xi-1}$ are pairwise disjoint;
\begin{eqnarray*}
\sum_{k=0}^{\xi-1}\dim V_{s+k}=\card(\msb_0\cup\ldots\cup\msb_{\xi-1}\cup\msc_{\xi-1})=
\card(\msb_0\cup\msc_0)\leq\card(\msc)
\end{eqnarray*}
and $\msb_0\cup\ldots\cup\msb_{\xi-1}\cup\msc_{\xi-1}$ is a generating set of $Q$. 

\vspace*{0,2cm}

\noindent \underline{Suppose that $\msb_1\subsetneq\msc_0$}.
For any $g\in\msc_0\setminus\msb_1$, with $\sord(g)=s+1$, if any, then
$g^{\sigma}=\sum_{l=1}^{d_{s+1}}\lambda_lf_{1,l}^{\sigma}$, for some
$\lambda_l\in\mbk$. Moreover, $g-\sum_{l=1}^{d_{s+1}}\lambda_lf_{1,l}\in Q$
and $\sord(g-\sum_{l=1}^{d_{s+1}}\lambda_lf_{1,l})>s+1$. Let
\begin{multline*}
\msc_{1,1}:=\left\{g\in\msc_0\setminus\msb_1\mid\sord(g)>s+1\right\}\subseteq\msc_0\setminus\msb_1
\mbox{ and }\\
\msc_{1,2}:=\left\{g-\sum_{l=1}^{d_{s+1}}\lambda_lf_{1,l}\bigg|\;
g\in\msc_0\setminus\msb_1\;, \sord(g)=s+1\;, 
g^{\sigma}=\sum_{l=1}^{d_{s+1}}\lambda_lf_{1,l}^{\sigma}\right\}\subset
\langle\msc\rangle_{\Bbbk}.
\end{multline*}
Take $\msc_1=\msc_{1,1}\cup\msc_{1,2}$. Let us prove that $\msc_1\neq\emptyset$. 
Indeed, since $\msb_1\subsetneq\msc_0$, then
there exists $g\in\msc_0\setminus\msb_1$, so $\sord(g)\geq s+1$; if $\sord(g)>s+1$, then $g\in\msc_{1,1}\subseteq\msc_1$; if $\sord(g)=s+1$, then 
$g-\sum_{l=1}^{d_{s+1}}\lambda_lf_{1,l}\in\msc_{1,2}\subseteq\msc_1$. In any case, $\msc_1\neq\emptyset$.

Note that the elements of $\msc_1$ have $\sigma$-order bigger than $s+1$, so
$\msb_0$, $\msb_1$ and $\msc_1$ are pairwise disjoint and 
$\card(\msb_0\cup\msb_1\cup\msc_1)=\card(\msb_0)+\card(\msb_1)+\card(\msc_1)$. As before, 
$\msc_{1,1}$ and $\msc_{1,2}$ are not necessarily disjoint sets, so
\begin{multline*}
\card(\msc_1)=\card(\msc_{1,1}\cup\msc_{1,2})\leq\card(\msc_{1,1})+\card(\msc_{1,2})\leq\\ 
\card(\{g\in\msc_0\setminus\msb_1\mid\sord(g)>s+1\})+
\card(\{g\in\msc_0\setminus\msb_1\mid\sord(g)=s+1\})=\\
\card(\msc_0\setminus\msb_1).
\end{multline*}
Therefore,
\begin{multline*}
d_s+d_{s+1}=\card(\msb_0)+\card(\msb_1)<
\card(\msb_0)+\card(\msb_1)+\card(\msc_1)=\\
\card(\msb_0\cup\msb_1\cup\msc_1)\leq
\card(\msb_0)+\card(\msb_1)+\card(\msc_0\setminus\msb_1)=
\card(\msb_0\cup\msc_0)\leq\card(\msc).
\end{multline*}

Since $\msb_0,\msb_1,\msc_{1,1},\msc_{1,2}\subset\langle\msc\rangle_{\Bbbk}$, 
then the ideal generated by $\msb_0\cup\msb_1\cup\msc_1$ is contained in the ideal generated by $\msc$, 
so $(\msb_0\cup\msb_1\cup\msc_1)\subseteq (\msc)=Q$. 
Let us prove that $Q\subseteq (\msb_0\cup\msb_1\cup\msc_1)$.
By {\sc Step $0$}, $Q=(\msb_0\cup\msc_0)$. 
Let $g\in\msb_0\cup\msc_0$. If $g\in\msb_0\cup\msb_1$ or
$g\in\msc_0\setminus\msb_1$ and $\sord(g)>s+1$, then, clearly, 
$g\in\msb_0\cup\msb_1\cup\msc_{1,1}\subset \msb_0\cup\msb_1\cup\msc_{1}$.
If $g\in\msc_0\setminus\msb_1$ and $\sord(g)=s+1$, then
$g=\sum_{l=1}^{d_{s+1}}\lambda_lf_{1,l}+(g-\sum_{l=1}^{d_{s+1}}\lambda_lf_{1,l})$
and $g\in\langle\msb_1\cup\msc_{1,2}\rangle_{\Bbbk}\subset (\msb_1\cup\msc_1)$. Therefore, 
$\msb_0\cup\msc_0\subset(\msb_0\cup\msb_1\cup\msc_1)$, $Q=(\msb_0\cup\msc_0)\subseteq
(\msb_0\cup\msb_1\cup\msc_1)$ and $\msb_0\cup\msb_1\cup\msc_1$ is 
a generating set of $Q$.

Summarizing, when $\msb_1\subsetneq\msc_0$, we obtain two disjoint sets 
$\msb_1=\{f_{1,1},\ldots,f_{1,d_{s+1}}\}$ and $\msc_1=\{g_{1,1},\ldots,g_{1,e_{s+1}}\}$, say, with
$\card(\msb_1)=d_{s+1}>0$ and $\card(\msc_1)=e_{s+1}>0$; $\msb_1$ and $\msc_1$ 
included in $\langle\msc\rangle_{\Bbbk}$; with $\msb_1^{\sigma}$ a basis of $V_{s+1}$ and
$\sord(g_{1,l})>s+1$;
$\msb_0$, $\msb_1$, $\msc_1$ pairwise disjoint
and such that $\msb_0\cup\msb_1\cup\msc_1$ is a generating set of $Q$ with
$d_s+d_{s+1}<\card(\msb_0\cup\msb_1\cup\msc_1)\leq\card(\msc)$. 

\vspace*{0,2cm} 

If $\xi=2$, we are done. Suppose that $2\leq \xi-1$, so in particular, $\xi\geq 3$.

\vspace*{0,2cm} 

\noindent \underline{\sc Step $i$, $2\leq i\leq \xi-1$: Construction of $\msb_i$ and $\msc_i$}.

\vspace*{0,2cm}

\noindent In particular, $\xi\geq i+1$. Suppose that we have built pairwise disjoint sets 
\begin{eqnarray*}
&&\msb_0,\msc_0\subset\langle\msc\rangle_{\Bbbk},\phantom{+}
\msb_0,\subsetneq\msc\;\mbox{ and }\;
\msb_0,\msc_0\neq\emptyset,\phantom{+}
\msb_0^{\sigma}\mbox{ a (non-empty) basis of }V_{s};\\
&&\msb_1,\msc_1\subset\langle\msc\rangle_{\Bbbk},\phantom{+}
\msb_1\subsetneq\msc_0\;\mbox{ and }\;\msc_1\neq\emptyset,\phantom{+}
\msb_1^{\sigma}\mbox{ a (possibly empty) basis of }V_{s+1};
\phantom{+}\ldots\phantom{+}\;\\
&&\msb_{i-1},\msc_{i-1}\subset\langle\msc\rangle_{\Bbbk},\phantom{+}
\msb_{i-1}\subsetneq\msc_{i-2}\;\mbox{ and }\;\msc_{i-1}\neq\emptyset,
\phantom{+}\msb_{i-1}^{\sigma}\mbox{ a (possibly empty) basis of }V_{s+i-1}.
\end{eqnarray*}
We suppose that $\msb_{i-1}\subsetneq\msc_{i-2}$ and $\msc_{i-1}\neq\emptyset$, 
because, otherwise, one takes
$\msb_i=\ldots=\msb_{\xi-1}=\emptyset$ and $\msc_{i-1}=\ldots=\msc_{\xi-1}=\emptyset$. Moreover, 
\begin{eqnarray*}
&&\msb_0=\{f_{0,1},\ldots,f_{0,d_s}\}\;\mbox{ and }\;
\msc_0=\{g_{0,1},\ldots,g_{0,e_s}\};\\
&&\mbox{If }\msb_1\neq\emptyset,\;\msb_1=\{f_{1,1},\ldots,f_{1,d_{s+1}}\}\;\mbox{ and }\;
\msc_1=\{g_{1,1},\ldots,g_{1,e_{s+1}}\};
\phantom{+}\ldots\phantom{+}\\
&&\mbox{If }\msb_{i-1}\neq\emptyset,\;\msb_{i-1}=\{f_{i-1,1},\ldots,
f_{i-1,d_{s+i-1}}\}\;\mbox{ and }
\msc_{i-1}=\{g_{i-1,1},\ldots,g_{i-1,e_{s+i-1}}\},
\end{eqnarray*}
where $\card(\msb_{k})=d_{s+k}=\dim V_{s+k}$, 
$\card(\msc_{k})=e_{s+k}>0$
and $\sord(g_{k,j})>s+k$, for all $k$ with $0\leq k\leq i-1$. 
Furthermore, $\msb_0,\ldots,\msb_{i-1},\msc_{i-1}$ are pairwise disjoint and
$\msb_0\cup\ldots\cup\msb_{i-1}\cup\msc_{i-1}$ is a generating set of $Q$ with
\begin{eqnarray*}
d_{s}+\cdots+d_{s+i-1}<\card(\msb_0\cup\ldots\cup\msb_{i-1}\cup\msc_{i-1})\leq\card(\msc).    
\end{eqnarray*}

\vspace*{0,2cm}

\noindent \underline{Suppose that $V_{s+i}=\{0\}$}. 
Then, take $\msb_i=\emptyset$ and $\msc_i=\msc_{i-1}$, so
$\msb_0,\ldots,\msb_{i},\msc_{i}$ are pairwise disjoint and
\begin{eqnarray*}
\msb_0\cup\ldots\cup\msb_{i-1}\cup\msb_{i}\cup\msc_{i}=\msb_0\cup\ldots \cup\msb_{i-1}\cup\msc_{i-1}, 
\end{eqnarray*}
which is a set of generators of $Q$. Moreover,
\begin{eqnarray*}
\sum_{k=0}^{i}d_{s+k}=\sum_{k=0}^{i-1}d_{s+k}<\card(\msb_0\cup\ldots\cup\msb_i\cup\msc_i)
\leq \card(\msc).
\end{eqnarray*}

\vspace*{0,2cm}

\noindent \underline{Suppose that $V_{s+i}\neq \{0\}$}.
Take $\{v_{i,1},\ldots,v_{i,d_{s+i}}\}\subset Q$ such that
$\{v^\sigma_{i,1},\ldots,v^{\sigma}_{i,d_{s+i}}\}$ is a basis of
$V_{s+i}$.  Since $v_{i,j}\in Q$ and 
$\msb_0\cup\ldots\cup\msb_{i-1}\cup\msc_{i-1}$ is a generating set of $Q$, 
then, for each $1\leq j\leq d_{s+i}$,
\begin{eqnarray}\label{eq-casi}
v_{i,j}=\sum_{k=0}^{i-1}\sum_{l=1}^{d_{s+k}}a^{k,l}_{i,j}f_{k,l}
+\sum_{l=1}^{e_{s+i-1}}b^{l}_{i,j}g_{i-1,l},
\end{eqnarray}
where $a^{k,l}_{i,j},b^l_{i,j}\in\mbk\lser\monx\rser$,
$f_{k,l}\in\msb_k$ and $g_{i-1,l}\in\msc_{i-1}$. Note that
$\sord(v_{i,j})=s+i$ and $\sord(f_{k,l})=s+k$, where $0\leq k\leq i-1$
and $1\leq l\leq d_{s+k}$. In addition, $\sord(g_{i-1,l})\geq s+i$,
for $1\leq l\leq\eta-e_{i-1}$. 
If $a^{k,l}_{i,j}$ is not a unit and $a^{k,l}_{i,j}\neq 0$, then
$\sord(a^{k,l}_{i,j}f_{k,l}) \geq\xi+s+k>s+i=\sord(v_{i,j})$, because $\xi\geq i+1$. Let
$\Lambda^0_{i,j}\subseteq \{1,\ldots,d_{s}\}$ denote the subset of
indices $l$ such that $a^{0,l}_{i,j}$ is invertible, with term of
order zero $\lambda^{0,l}_{i,j}\in\mbk$, $\lambda^{0,l}_{i,j}\neq 0$,
and suppose that $\Lambda^0_{i,j}\neq\emptyset$. Taking
$\sigma$-leading forms in Equality~\eqref{eq-casi}, we deduce that
$\sum_{l\in\Lambda^0_{i,j}}\lambda^{0,l}_{i,j}f_{0,l}^{\sigma}=0$. Since
$\{f^{\sigma}_{0,1},\ldots,f^{\sigma}_{0,d_{s}}\}$ is a basis of
$V_s$, then $\lambda^{0,l}_{i,j}=0$, a contradiction. Therefore, all
the $a^{0,l}_{i,j}$ are non-units. For all $k=1,\ldots,i-1$,
$\sord(f_{k,l})=s+k<s+i=\sord(v_{i,j})$. Thus, similarly, and
recursively on $k=1,\ldots,i-1$, one can prove that $a^{k,l}_{i,j}$
are not invertible. Hence,
$\sord(a^{k,l}_{i,j}f_{k,l})>s+i=\sord(v_{i,j})$, for all
$k=0,\ldots,i-1$ and for all $l=1,\ldots,d_{s+k}$.  Let
$\Lambda_{i,j}\subseteq \{1,\ldots,e_{s+i-1}\}$ be the set of
indices $l$ such that $b^l_{i,j}$ is invertible, with term of order
zero $\mu^{l}_{i,j}\in\mbk$, $\mu^l_{i,j}\neq 0$, and such that
$\sord(g_{i-1,l})=s+i$. Note that $\Lambda_{i,j}\neq\emptyset$ for, if
$b_{i,j}^l$ is not a unit, then $\sord(b_{i,j}^lg_{i-1,l})\geq
\xi+s+i>s+i=\sord(v_{i,j})$. Taking $\sigma$-leading forms in
Equality~\eqref{eq-casi}, we deduce that
\begin{eqnarray*}
v_{i,j}^{\sigma}=\sum_{l\in\Lambda_{i,j}}\mu^l_{i,j}g_{i-1,l}^{\sigma}.
\end{eqnarray*}
Let $\Lambda_i=\Lambda_{i,1}\cup\ldots\cup\Lambda_{i,d_{s+i}}$. Then
\begin{eqnarray*}
V_{s+i}=\langle
v_{i,1}^{\sigma},\ldots,v_{i,d_{s+i}}^{\sigma}\rangle\subseteq \langle
g_{i-1,l}^{\sigma}\mid l\in\Lambda_i\rangle\subseteq V_{s+i}.
\end{eqnarray*}
So $V_{s+i}=\langle g_{i-1,l}^{\sigma}\mid l\in\Lambda_i\rangle$.
Since $\dim V_{s+i}=d_{s+i}$ and $\{g_{i-1,l}^{\sigma}\mid l\in\Lambda_i\}$
spans $V_{s+i}$, then $\card(\Lambda_i)\geq d_{s+i}$, and there exists
a subset $\Gamma\subseteq\Lambda_i\subseteq\{1,\ldots,e_{s+i-1}\}$,
$\card(\Gamma)=d_{s+i}$, such that $\{g_{i-1,l}^{\sigma}\mid
l\in\Gamma\}$ is a basis of $V_{s+i}$. In particular,
$d_{s+i}=\card(\Gamma)\leq\card(\Lambda_i)\leq e_{s+i-1}$. 
Renaming the elements of
$\msc_{i-1}=\{g_{i-1,1},\ldots,g_{i-1,e_{s+i-1}}\}$, we can suppose
that $\{g^{\sigma}_{i-1,1},\ldots,g^{\sigma}_{i-1,d_{s+i}}\}$ is a
basis of $V_{s+i}$. For $l=1,\ldots,d_{s+i}$, set $f_{i,l}:=g_{i-1,l}$
and let $\msb_i=\{f_{i,1},\ldots,f_{i,d_{s+i}}\}$. Note that
$\msb_i\subseteq\msc_{i-1}\subset \langle\msc\rangle_{\Bbbk}$. 

\vspace*{0,2cm}

Once $\msb_i$ has been constructed, let us build $\msc_i$.

\vspace*{0,2cm}

\noindent \underline{Suppose that $\msb_i=\msc_{i-1}$}. Then,
$\msb_0\cup\ldots\msb_{i-1}\cup\msb_i=
\msb_0\cup\ldots\msb_{i-1}\cup\msc_{i-1}$, which, by 
{\sc Step $i-1$}, is a generating set of
$Q$. Let us show that $V_r=\{0\}$, for all
$s+i<r<s+\xi$. If $\xi=i+1$, there is nothing to prove. Suppose that $\xi\geq i+2$,
and that there exists $g\in V_r$, $g\neq 0$, for some
$s+i<r<s+\xi$. Then $g=f^\sigma$, with $f\in Q$, $\sord(f)=r$. Since
$f\in Q=(\msb_0\cup\ldots\cup\msb_i)$, then
\begin{eqnarray*}
f=\sum_{k=0}^{i}\sum_{l=1}^{d_{s+k}}a_{k,l}f_{k,l},
\end{eqnarray*}
where $a_{k,l}\in\mbk\lser\monx\rser$. Let $\Lambda_0\subseteq
\{1,\ldots,d_s\}$ denote the subset of indices $l$ such that $a_{0,l}$
is invertible, with term of order zero $\lambda_{0,l}\in\mbk$,
$\lambda_{0,l}\neq 0$. Suppose that $\Lambda_0\neq\emptyset$. Since
$\sord(f_{0,l})=s<s+i<r=\sord(f)$, then, necessarily,
$\sum_{l\in\Lambda_0}\lambda_{0,l}f_{0,l}^{\sigma}=0$. Since
$\{f_{0,1}^{\sigma},\ldots,f_{0,d_{s}}^{\sigma}\}$ is a basis of
$V_s$, this implies $\lambda_{0,l}=0$, a contradiction. Therefore,
$\Lambda_0=\emptyset$, all $a_{0,l}$ are non-units
and $\sord(a_{0,l}f_{0,l})\geq \xi+s\geq s+i+2$. 
Let $\Lambda_1\subseteq
\{1,\ldots,d_{s+1}\}$ denote the subset of indices $l$ such that $a_{1,l}$
is invertible, with term of order zero $\lambda_{1,l}\in\mbk$,
$\lambda_{1,l}\neq 0$. Suppose that $\Lambda_1\neq\emptyset$. Since
$\sord(f_{1,l})=s+1<r=\sord(f)$ and
$\sord(f_{1,l})=s+1<s+i+2\leq s+\xi\leq\sord(a_{0,l'}f_{0,l'})$, then, necessarily,
$\sum_{l\in\Lambda_1}\lambda_{1,l}f_{1,l}^{\sigma}=0$. Since
$\{f_{1,1}^{\sigma},\ldots,f_{1,d_{s+1}}^{\sigma}\}$ is a basis of
$V_{s+1}$, this implies $\lambda_{1,l}=0$, a contradiction. Therefore,
$\Lambda_1=\emptyset$, all $a_{1,l}$ are non-units and 
$\sord(a_{1,l}f_{1,l})\geq \xi+s+1\geq s+i+3$. Proceed by induction. 
Let $\Lambda_{i}\subseteq
\{1,\ldots,d_{s+i}\}$ denote the subset of indices $l$ such that $a_{i,l}$
is invertible, with term of order zero $\lambda_{i,l}\in\mbk$,
$\lambda_{i,l}\neq 0$. Suppose that $\Lambda_i\neq\emptyset$. Since
$\sord(f_{i,l})=s+i<r=\sord(f)$ and
$\sord(f_{i,l})=s+i<s+i+2\leq\sord(a_{0,l'}f_{0,l'}),\ldots, 
\sord(a_{i-1,l'}f_{i-1,l'})$, then, necessarily,
$\sum_{l\in\Lambda_i}\lambda_{i,l}f_{i,l}^{\sigma}=0$. Since
$\{f_{i,1}^{\sigma},\ldots,f_{i,d_{s+1}}^{\sigma}\}$ is a basis of
$V_{s+i}$, this implies $\lambda_{i,l}=0$, a contradiction. Therefore,
$\Lambda_i=\emptyset$ and all the $a_{i,l}$ are non-units. Thus, 
$a_{k,l}$ is not invertible, for all $k=1,\ldots,i$. But then,
$\sord(a_{k,l}f_{k,l})\geq \xi+s>r$, whereas $\sord(f)=r$, a
contradiction again. Therefore, $f$ must be zero and 
$V_{s+i+1}=\ldots=V_{s+\xi-1}=\{0\}$.

So, when $\msb_i=\msc_{i-1}$, take $\msb_{i+1}=\ldots=\msb_{\xi-1}=\emptyset$ and 
$\msc_i=\msc_{i+1}=\ldots=\msc_{\xi-1}=\emptyset$. Then, 
$\msb_0,\ldots,\msb_{\xi-1},\msc_{\xi-1}$ 
satisfy the required properties, namely, 
$\msb_0,\ldots,\msb_{\xi-1},\msc_{\xi-1}$ are included in $\langle \msc\rangle_{\Bbbk}$,
$\msb_k^{\sigma}$ is a (possibly empty) basis of $V_{s+k}$,
$\msb_0,\ldots,\msb_{\xi-1},\msc_{\xi-1}$ are pairwise disjoint,
\begin{multline*}
\sum_{k=0}^{\xi-1}\dim V_{s+k}=\sum_{k=0}^{i}\dim V_{s+k}=
\card(\msb_0\cup\ldots\cup\msb_{i})=\\
\card(\msb_0\cup\ldots\cup\msb_{\xi-1}\cup\msc_{\xi-1})=
\card(\msb_0\cup\ldots\cup\msb_{i-1}\cup\msc_{i-1})\leq\card(\msc).
\end{multline*}
and $\msb_0\cup\ldots\cup\msb_{\xi-1}\cup\msc_{\xi-1}=
\msb_0\cup\ldots\msb_{i-1}\cup\msc_{i-1}$ is a generating set of $Q$. 

\vspace*{0,2cm}

\noindent \underline{Suppose that $\msb_i\subsetneq\msc_{i-1}$}. 
Let $g\in\msc_{i-1}\setminus\msb_i$, in particular, $\sord(g)\geq s+i$. If
$\sord(g)=s+i$, then $g^{\sigma}=
\sum_{l=1}^{d_{s+i}}\lambda_{l}f_{i,l}^{\sigma}$, for some
$\lambda_{l}\in\mbk$. Therefore, $g-\sum_{l=1}^{d_{s+i}}\lambda_{l}f_{i,l}\in Q$
and $\sord(g-\sum_{l=1}^{d_{s+i}}\lambda_{l}f_{i,l})>s+i$. Let
\begin{multline*}
\msc_{i,1}:=\left\{g\in\msc_{i-1}\setminus\msb_i\mid\sord(g)>s+i\right\}
\subseteq\msc_{i-1}\setminus\msb_i\mbox{ and }\\
\msc_{i,2}:=\left\{g-\sum_{l=1}^{d_{s+i}}\lambda_lf_{i,l}\bigg|\;
g\in\msc_{i-1}\setminus\msb_i\;, \sord(g)=s+i\;, 
g^{\sigma}=\sum_{l=1}^{d_{s+i}}\lambda_lf_{i,l}^{\sigma}\right\}\subset
\langle\msc\rangle_{\Bbbk}.
\end{multline*}
Take $\msc_i=\msc_{i,1}\cup\msc_{i,2}$. Since $\msb_i\subsetneq\msc_{i-1}$, then
there exists $g\in\msc_{i-1}\setminus\msb_i$, so 
$\sord(g)\geq s+i$; if $\sord(g)>s+i$, then $g\in\msc_{i,1}\subseteq\msc_{i}$; if $\sord(g)=s+i$, 
then $g-\sum_{l=1}^{d_{s+i}}\lambda_lf_{i,l}\in\msc_{i,2}\subseteq\msc_{i}$. In any case, $\msc_i\neq\emptyset$.

Note that the elements of $\msc_i$ have $\sigma$-order bigger than $s+i$, so
$\msb_0,\ldots,\msb_{i-1},\msb_i$ and $\msc_i$ are pairwise disjoint and 
$\card(\msb_0\cup\ldots\cup\msb_i\cup\msc_i)=
\card(\msb_0)+\cdots+\card(\msb_i)+\card(\msc_i)$. Observe that  
$\msc_{1,1}$ and $\msc_{1,2}$ are not necessarily disjoint sets, so
\begin{multline*}
\card(\msc_i)=\card(\msc_{i,1}\cup\msc_{i,2})\leq\card(\msc_{i,1})+\card(\msc_{i,2})\leq\\
\card(\{g\in\msc_{i-1}\setminus\msb_i\mid\sord(g)>s+i\})+
\card(\{g\in\msc_{i-1}\setminus\msb_i\mid\sord(g)=s+i\})=\\
\card(\msc_{i-1}\setminus\msb_i).
\end{multline*}
Therefore,
\begin{multline*}
\sum_{k=0}^{s+i}d_{s+k}=\sum_{k=0}^{s+i}\card(\msb_k)<
\sum_{k=0}^{s+i}\card(\msb_k)+\card(\msc_i)=\\
\card(\msb_0\cup\ldots\cup\msb_i\cup\msc_i)
\leq\sum_{k=0}^{s+i-1}\card(\msb_k)+\card(\msb_i)+
\card(\msc_{i-1}\setminus\msb_i)=\\\sum_{k=0}^{s+i-1}\card(\msb_k)+\card(\msc_{i-1})=
\card(\msb_0\cup\ldots\cup\msb_{i-1}\cup\msc_{i-1})\leq\card(\msc).
\end{multline*}
Since, $\msb_0,\ldots,\msb_i,\msc_{i,1},\msc_{i,2}\subset\langle\msc\rangle_{\Bbbk}$, 
then the ideal generated by
$\msb_0\cup\ldots\cup\msb_i\cup\msc_i\subseteq (\msc)$, so 
$(\msb_0\cup\ldots\cup\msb_i\cup\msc_i)\subseteq (\msc)=Q$. 
Let us prove that $Q\subseteq (\msb_0\cup\ldots\cup\msb_i\cup\msc_i)$.
By {\sc Step $i-1$}, 
$Q=(\msb_0\cup\ldots\cup\msb_{i-1}\cup\msc_{i-1})$. 
Let $g\in\msb_0\cup\ldots\cup\msb_{i-1}\cup\msc_{i-1}$. 
If $g\in\msb_0\cup\ldots\cup\msb_{i}$ or
$g\in\msc_{i-1}\setminus\msb_i$ and $\sord(g)>s+i$, then, clearly, 
$g\in\msb_0\cup\ldots\cup\msb_i\cup\msc_{i,1}\subset \msb_0\cup\ldots\cup\msb_i\cup\msc_{i}$.
If $g\in\msc_{i-1}\setminus\msb_i$ and $\sord(g)=s+i$, then
$g=\sum_{l=1}^{d_{s+i}}\lambda_lf_{i,l}+(g-\sum_{l=1}^{d_{s+i}}\lambda_lf_{i,l})$
and $g\in\langle\msb_i\cup\msc_{i,2}\rangle_{\Bbbk}\subset (\msb_i\cup\msc_i)$. Therefore, 
$Q=(\msb_0\cup\ldots\cup\msb_{i-1}\cup\msc_{i-1})\subseteq(\msb_0\cup\ldots\cup\msb_i\cup\msc_i)$
and $\msb_0\cup\ldots\cup\msb_i\cup\msc_i$ is a set of generators of $Q$.

Summarizing, when $\msb_i\subsetneq\msc_{i-1}$, we obtain two disjoint sets 
$\msb_i=\{f_{i,1},\ldots,f_{i,d_{s+i}}\}$ and $\msc_i=\{g_{i,1},\ldots,g_{i,e_{s+i}}\}$, say, with
$\card(\msb_i)=d_{s+i}>0$ and $\card(\msc_i)=e_{s+i}>0$, $\msb_i$ and $\msc_i$ 
included in $\langle\msc\rangle_{\Bbbk}$; with $\msb_i^{\sigma}$ a basis of $V_{s+i}$
and $\sord(g_{i,l})>s+i$, $\msb_0,\ldots,\msb_i,\msc_i$ pairwise disjoint
and such that $\msb_0\cup\ldots\cup\msb_i\cup\msc_i$ is a generating set of $Q$ with
$\sum_{k=0}^id_{s+k}<\card(\msb_0\cup\ldots\cup\msb_i\cup\msc_i)\leq\card(\msc)$. 

\vspace*{0,2cm}

In the end, we deduce the existence of $\msb_0,\ldots,\msb_{\xi-1},\msc_{\xi-1}$,
a family of pairwise disjoint subsets of $\langle \msc\rangle_{\Bbbk}$;
$\msb_i^\sigma$ a (possible empty) basis of $V_{s+i}$
and, if $\msc_{\xi-1}\neq\emptyset$ and $g\in\msc_{\xi-1}$, then $\sord(g)\geq s+\xi$;
where $\msb_0\cup\ldots\cup\msb_{\xi-1}\cup\msc_{\xi-1}$ 
is a generating set of $Q$ and 
\begin{eqnarray*}
\sum_{i=0}^{\xi-1}\dim V_{s+i}\leq\card(\msb_0\cup\ldots\cup\msb_{\xi-1}
\cup\msc_{\xi-1})\leq\card(\msc).
\end{eqnarray*}
Since $\mbk\lser\monx \rser$ is Noetherian local, 
all the minimal generating sets of $Q$ have the same cardinality 
(see, e.g., \cite[Theorem~2.3]{matsumura}). Therefore, 
if $\msc$ is a minimal generating set of $Q$, since
$\msb_0\cup\ldots\cup\msb_{\xi-1}\cup\msc_{\xi-1}$ is a generating set of $Q$ whose cardinality is smaller than or equal to the cardinality of $\msc$, one deduces
that $\msb_0\cup\ldots\cup\msb_{\xi-1}\cup\msc_{\xi-1}$ is also
a minimal generating set of $Q$ of cardinality equals to 
$\card(\msc)=\mu(Q)$, which concludes the proof. 
\end{proof}

\vspace*{0,2cm}

\begin{proof}[Proof of Theorem~\ref{moh43-2}] 
Let $\msd_0,\ldots,\msd_{\xi-1}\subset Q$ such that
$\msd_i=\emptyset$ or $\msd_i^{\sigma}$ is a linearly independent subset of $V_{s+i}$. 
In this last case, set $\msd_i=\{u_{i,1},\ldots,u_{i,\delta_i}\}$, 
where $\card(\msd_i)=\delta_i$.
Let $\msc$ be a minimal generating set of $Q$. Let 
$\msb_0,\ldots,\msb_{\xi-1},\msc_{\xi-1}\subset\langle\msc\rangle_{\Bbbk}$ 
be defined as in Theorem~\ref{moh43-1}. In particular, when they are not empty, set $\msb_i=\{f_{i,1},\ldots,f_{i,d_{s+i}}\}$ and
$\msc_{\xi-1}=\{g_{\xi-1,1},\ldots,g_{\xi-1,e_{s+\xi-1}}\}$,
where $\msb_i^\sigma$ is a basis of
$V_{s+i}$, $\sord(g_{\xi-1,l})>s+\xi-1$,
and $\msb_0\cup\ldots\cup\msb_{\xi-1}\cup\msc_{\xi-1}$
is a minimal generating set of $Q$. 
Let us iteratively construct, for each $i=0,\ldots,\xi-1$, 
$\msb^{\prime}_i$ satisfying the following list of conditions $\mbl_i$.
\begin{itemize}
\item[$\bullet$] $\msb^{\prime}_i$ is a subset of $\msb_i$.
\item[$\bullet$] 
$\msd_i^{\sigma}\cup(\msb_i^{\prime})^{\sigma}$ is a basis of $V_{s+i}$,
understanding that $\msd_i^{\sigma}=\emptyset$ if $\msd_i=\emptyset$,
or $(\msb_i^{\prime})^{\sigma}=\emptyset$ if $\msb_i^{\prime}=\emptyset$.
\item[$\bullet$]  $\msd_0,\msb^{\prime}_0,
\ldots,\msd_{i},\msb^{\prime}_i,\msb_{i+1},\ldots,
\msb_{\xi-1},\msc_{\xi-1}$ are pairwise disjoint subsets of $Q$.
\item[$\bullet$] $(\msd_0\cup\msb^{\prime}_0)\cup
\ldots\cup(\msd_{i}\cup\msb^{\prime}_i)\cup\msb_{i+1}
\cup\ldots\cup \msb_{\xi-1}\cup\msc_{\xi-1}$ 
is a minimal generating set of $Q$.
\end{itemize}
In the end, once all the $\msb^{\prime}_i$ are constructed satisfying the conditions
$\mbl_i$, we deduce that
\begin{eqnarray*}
(\msd_0\cup\msb^{\prime}_0)\cup
\ldots\cup(\msd_{\xi-1}\cup\msb^{\prime}_{\xi-1})\cup\msc_{\xi-1}
\end{eqnarray*}
is a minimal generating set of $Q$. In particular, 
$\msd_0\cup\ldots\cup\msd_{\xi-1}$ can be extended to a minimal generating set of $Q$. 
The proof basically consists on applying repeatedly the
Steinitz Exchange Lemma.

\vspace*{0,2cm}

\noindent \underline{\sc Step $0$: Construction of $\msb_0^{\prime}$}.

\vspace*{0,2cm} 

If $\msd_0=\emptyset$, let $\msb_0^{\prime}=\msb_0$, which clearly satisfies 
the conditions $\mbl_0$. Suppose that 
$\msd_0\neq \emptyset$. Take $u_{0,1}\in\msd_0\subset Q$. Since 
$\msb_0\cup\ldots\cup\msb_{\xi-1}\cup\msc_{\xi-1}$ is a generating set of the ideal $Q$,
it follows that
\begin{eqnarray*}\label{eq-casi2}
u_{0,1}=\sum_{k=0}^{\xi-1}\sum_{l=1}^{d_{s+k}}a^{k,l}_{0,1}f_{k,l}
+\sum_{l=1}^{e_{s+\xi-1}}b^{l}_{0,1}g_{\xi-1,l},
\end{eqnarray*}
where $a^{k,l}_{0,1},b^l_{0,1}\in\mbk\lser\monx\rser$. Note that, if $a_{0,1}^{0,l}$
is not a unit and $a_{0,1}^{0,l}\neq 0$, then $\sord(a_{0,1}^{0,l}f_{0,l})\geq \xi+s>s=\sord(u_{0,1})$. Moreover, $\sord(g_{\xi-1,l})\geq s+\xi>s=\sord(u_{0,1})$. 
In particular, the subset  
$\Lambda_{0,1}\subseteq \{1,\ldots,d_{s}\}$ of
indices $l$ such that $a^{0,l}_{0,1}$ is invertible, with term of
order zero $\lambda^{0,l}_{0,1}\in\mbk$, $\lambda^{0,l}_{0,1}\neq 0$, is not empty. 
Taking $\sigma$-leading forms in both parts of the equality, one deduces that
$u_{0,1}^{\sigma}=\sum_{l\in\Lambda_{0,1}}\lambda^{0,l}_{0,1}f_{0,l}^{\sigma}$.
Renaming the elements of $\msb_0$, we can suppose that
$1\in\Lambda_{0,1}$, so $\lambda_{0,1}^{0,1}$ and $a_{0,1}^{0,1}$ are
invertible. Set $\msb_{0,1}:=\{u_{0,1},f_{0,2},\ldots,f_{0,d_s}\}$. Clearly,
$f_{0,1}^{\sigma}$ is in the vector space spanned by $\msb_{0,1}^{\sigma}$
and $\msb_{0,1}^{\sigma}$ is a basis of $V_s$. Moreover, since $a_{0,1}^{0,1}$ is invertible,
$f_{0,1}$ is in the ideal generated by $\msb_{0,1}\cup
\msb_{1}\cup\ldots\cup\msb_{\xi-1}\cup\msc_{\xi-1}$. Thus,
$\msb_{0,1}\cup\msb_{1}\cup\ldots\cup\msb_{\xi-1}\cup\msc_{\xi-1}$
is a minimal generating set of $Q$. 

If $\delta_0=1$, i.e., $\msd_0=\{u_{0,1}\}$, take $\msb_0^{\prime}=\msb_{0,1}\setminus\msd_0$, 
which satisfies the conditions $\mbl_0$. 

Suppose that $\delta_0\geq 2$. So, let $u_{0,2}\in\msd_0\subset Q$. Then,
\begin{eqnarray*}
u_{0,2}=a_{0,2}^{0,1}u_{0,1}+\sum_{l=2}^{d_s}a_{0,2}^{0,l}f_{0,l}+
\sum_{k=1}^{\xi-1}\sum_{l=1}^{d_{s+k}}a^{k,l}_{0,2}f_{k,l}
+\sum_{l=1}^{e_{s+\xi-1}}b^{l}_{0,2}g_{\xi-1,l},
\end{eqnarray*}
where $a^{k,l}_{0,2},b^l_{0,2}\in\mbk\lser\monx\rser$. Note that, if $a_{0,2}^{0,l}$
is not a unit and $a_{0,2}^{0,l}\neq 0$, then $\sord(a_{0,2}^{0,l}f_{0,l})\geq \xi+s>s=\sord(u_{0,2})$. Moreover, $\sord(g_{\xi-1,l})\geq s+\xi>s=\sord(u_{0,1})$. 
Let $\Lambda_{0,2}\subseteq \{2,\ldots,d_{s}\}$ be the subset of
indices $l$ such that $a^{0,l}_{0,2}$ is invertible, 
with term of order zero $\lambda_{0,2}^{0,l}\in\mbk$, 
$\lambda^{0,l}_{0,2}\neq 0$. 
If $\Lambda_{0,2}$ were empty, on taking
$\sigma$-orders in the equality above, one would deduce that
$u_{0,2}^{\sigma}=\lambda_{0,2}^{0,1}u_{0,1}^{\sigma}$, a
contradiction because $\msd_{0}^{\sigma}$ is a linearly independent subset of $V_s$. 
Thus, $\Lambda_{0,2}\neq\emptyset$ and
$u_{0,2}^{\sigma}=\lambda_{0,2}^{0,1}u_{0,1}^{\sigma}+
\sum_{l\in\Lambda_{0,2}}\lambda^{0,l}_{0,2}f_{0,l}^{\sigma}$.
Renaming the elements of $\msb_{0,1}\setminus\{u_{0,1}\}$, we can suppose that
$2\in\Lambda_{0,2}$, so $\lambda_{0,2}^{0,2}$ and $a_{0,2}^{0,2}$ are
invertible. Set $\msb_{0,2}:=\{u_{0,1},u_{0,2},f_{0,3},\ldots,
f_{0,d_s}\}$. Clearly, $f_{0,2}^{\sigma}$ is in the vector space spanned by
$\msb_{0,2}^{\sigma}$ and $\msb_{0,2}^{\sigma}$ is a basis of $V_s$. Furthermore,
since $a_{0,2}^{0,2}$ is invertible, it follows that $f_{0,2}$ is in the ideal generated by
$\msb_{0,2}\cup\msb_{1}\cup\ldots\cup\msb_{\xi-1}\cup\msc_{\xi-1}$. Thus, 
$\msb_{0,2}\cup\msb_{1}\cup\ldots\cup\msb_{\xi-1}\cup\msc_{\xi-1}$
is a minimal generating set of $Q$. 

If $\delta_0=2$, i.e., $\msd_0=\{u_{0,1},u_{0,2}\}$, take 
$\msb_0^{\prime}=\msb_{0,2}\setminus\msd_0$, 
which clearly satisfies the conditions $\mbl_0$.

If $\delta_0>2$, proceed recursively to substitute all the elements of $\msd_0$ in $\msb_0$. After some renaming, one gets $\msb_{0,\delta_0}:=\{u_{0,1},\ldots,u_{0,\delta_0},f_{0,\delta_0+1},\ldots,f_{0,d_s}\}$, such that $\msb_{0,\delta_0}\cup\msb_{1}\cup\ldots\cup\msb_{\xi-1}\cup\msc_{\xi-1}$ is a minimal generating set of $Q$. Then, the set $\msb_0^{\prime}:=\msb_{0,\delta_0}\setminus\msd_0$ 
satisfies the conditions $\mbl_0$. For further convenience, call
$u_{0,\delta_0+1},\ldots,u_{0,d_s}$ to the elements of $\msb_0^{\prime}$.

\vspace*{0,2cm}

\noindent \underline{Step $i$, $1\leq i\leq \xi-1$: Construction of $\msb_{i}^{\prime}$.}

\vspace*{0,2cm} 

In particular, $\xi\geq i+1$. By the step $i-1$, there exist a family of subsets
$\msb_0^{\prime},\ldots,\msb_{i-1}^{\prime}$, each $\msb_k^{\prime}$ contained in $\msb_k$, 
satisfying the conditions of $\mbl_{i-1}$. That is,
$\msd_k^{\sigma}\cup(\msb_k^{\prime})^{\sigma}$ is a (possibly empty) basis of $V_{s+k}$, for $k=0,\ldots,i-1$; the family of subsets of $Q$, 
$\msd_0,\msb^{\prime}_0,
\ldots,\msd_{i-1},\msb^{\prime}_{i-1},\msb_{i},\ldots,
\msb_{\xi-1},\msc_{\xi-1}$ are pairwise disjoint and their union
$(\msd_0\cup\msb^{\prime}_0)\cup\ldots\cup(\msd_{i-1}\cup\msb^{\prime}_{i-1})\cup\msb_{i}
\cup\ldots\cup \msb_{\xi-1}\cup\msc_{\xi-1}$ is a minimal generating set of $Q$. 
For the sake of convenience, set $\msb_k^{\prime}=\{u_{k,\delta_{k}+1},\ldots,u_{k,d_{s+k}}\}$, 
$0\leq k\leq i-1$.

If $\msd_i=\emptyset$, let $\msb_i^{\prime}=\msb_i$, which satisfies the conditions $\mbl_i$.
Suppose that $\msd_i\neq\emptyset$. Take $u_{i,1}\in\msd_i\subset Q$. Since 
$(\msd_0\cup\msb^{\prime}_0)\cup\ldots\cup(\msd_{i-1}\cup\msb^{\prime}_{i-1})\cup\msb_{i}
\cup\ldots\cup \msb_{\xi-1}\cup\msc_{\xi-1}$ is a minimal generating set of $Q$, then 
\begin{eqnarray}\label{eq-d}
u_{i,1}=\sum_{k=0}^{i-1}\sum_{l=1}^{d_{s+k}}a^{k,l}_{i,1}u_{k,l}+
\sum_{k=i}^{\xi-1}\sum_{l=1}^{d_{s+k}}a^{k,l}_{i,1}f_{k,l}
+\sum_{l=1}^{e_{s+\xi-1}}b^{l}_{i,1}g_{\xi-1,l},
\end{eqnarray}
where $a^{k,l}_{i,1},b^l_{i,1}\in\mbk\lser\monx\rser$, $f_{k,l}\in\msb_k$, 
$g_{\xi-1,l}\in\msc_{\xi-1}$.  
If $a_{i,1}^{k,l}$ is not a unit and $a_{i,1}^{k,l}\neq 0$, 
then $\sord(a_{i,1}^{k,l}u_{k,l})\geq \xi+s+k>s+i=\sord(u_{i,1})$, $0\leq k\leq i-1$. 
Let $\Lambda_{i,1}^0\subseteq \{1,\ldots,d_{s}\}$ denote the subset of
indices $l$ such that $a^{0,l}_{i,1}$ is invertible, with term of
order zero $\lambda^{0,l}_{i,1}\in\mbk$, $\lambda^{0,l}_{i,1}\neq 0$.
Suppose that $\Lambda_{i,1}^0\neq\emptyset$. Taking $\sigma$-leading
forms in Equality~\eqref{eq-d}, we deduce that
$\sum_{l\in\Lambda_{i,1}^0}\lambda^{0,l}_{i,1}u_{0,l}^{\sigma}=0$. Since
$\msd_0^\sigma\cup(\msb_0^{\prime})^\sigma$ is a basis of $V_s$, this implies that
$\lambda^{0,l}_{i,1}=0$, which is a contradiction. Therefore, 
$a_{i,1}^{0,l}$ are non-units, for all $l=1,\ldots,d_s$. Proceed by induction 
on $k=1,\ldots,i-1$. 
Let $\Lambda_{i,1}^k\subseteq \{1,\ldots,d_{s+k}\}$ denote the subset of
indices $l$ such that $a^{k,l}_{i,1}$ is invertible, with term of
order zero $\lambda^{k,l}_{i,1}\in\mbk$, $\lambda^{k,l}_{i,1}\neq 0$.
Suppose that $\Lambda_{i,1}^k\neq\emptyset$. Taking $\sigma$-leading
forms in Equality~\eqref{eq-d}, we deduce that
$\sum_{l\in\Lambda_{i,1}^k}\lambda^{k,l}_{i,1}u_{k,l}^{\sigma}=0$. Since
$\msd_k^\sigma\cup(\msb_k^{\prime})^\sigma$ is a basis of $V_{s+k}$, this implies that
$\lambda^{k,l}_{i,1}=0$, which is a contradiction. Therefore, 
$a_{i,1}^{k,l}$ are non-units, for all $l=1,\ldots,d_{s+k}$. 
Similarly, one proves that $a_{i,1}^{k,l}$ are non-units, for all $k=1,\ldots,i-1$ and
all $l=1,\ldots,d_{s+k}$.

Let now $\Lambda_{i,1}^i\subseteq
\{1,\ldots,d_{s+i}\}$ be the subet of indices $l$ such that
$a_{i,1}^{i,l}$ is invertible, with term of order zero
$\lambda_{i,1}^{i,l}\in\mbk$, $\lambda_{i,1}^{i,l}\neq 0$. Taking
$\sigma$-leading forms, one deduces that
$u_{i,1}^{\sigma}=\sum_{l\in\Lambda_{i,1}^i}\lambda_{i,1}^{i,l}f_{i,1}^{\sigma}$.
Since $\msb_i^\sigma=\{f_{i,1}^{\sigma},\ldots,f_{i,d_{s+i}}^{\sigma}\}$ is a basis of $V_{s+i}$ and 
$u_{i,1}^{\sigma}\in V_{s+i}$, $u_{i,1}^\sigma\neq 0$, 
renaming the elements of $\msb_i$, one can suppose that
$1\in\Lambda_{i,1}^{i}$, so $\lambda_{i,1}^{i,1}$ and
$a_{i,1}^{i,1}$ are invertible. Set
$\msb_{i,1}=\{u_{i,1},f_{i,2},\ldots,f_{i,d_{s+i}}\}$. Then,
$f_{i,1}^{\sigma}$ is in the vector space spanned by $\msb_{i,1}^{\sigma}$ and
$\msb_{i,1}^{\sigma}$ is a basis of $V_{s+i}$. Moreover, since $a_{i,1}^{i,1}$ is invertible, 
$f_{i,1}$ is in the ideal generated by 
\begin{eqnarray*}
(\msd_0\cup\msb_0^{\prime})\cup\ldots\cup(\msd_{i-1}\cup\msb_{i-1}^{\prime})\cup
\msb_{i,1}\cup\msb_{i+1}\cup\ldots\cup\msb_{\xi-1}\cup\msc_{\xi-1},
\end{eqnarray*}
which is a minimal generating set for $Q$. 

If $\delta_i=1$, i.e., $\msd_i=\{u_{i,1}\}$, take 
$\msb_i^{\prime}=\msb_{i,1}\setminus\msd_i$, which
satisfies $\mbl_i$. Suppose that $\delta_i\geq 2$. Let $u_{i,2}\in\msd_i\subset Q$. Then, 
similarly as before, write 
\begin{eqnarray*}
u_{i,2}=\sum_{k=0}^{i-1}\sum_{l=1}^{d_{s+k}}a^{k,l}_{i,2}u_{k,l}+
a_{i,2}^{i,1}u_{i,1}+\sum_{l=2}^{d_{s+i}}a_{i,2}^{i,l}f_{i,l}+
\sum_{k=i+1}^{\xi-1}\sum_{l=1}^{d_{s+k}}a^{k,l}_{i,2}f_{k,l}
+\sum_{l=1}^{e_{s+\xi-1}}b^{l}_{i,2}g_{\xi-1,l}.
\end{eqnarray*}
As before, one first deduces that $a_{i,2}^{k,l}$ are non-units, for $k=0,\ldots,i-1$, and 
$l=1,\ldots,d_{s+k}$. So, $\sord(a_{i,2}^{k,l}u_{k,l})\geq \xi+s+k>s+i=\sord(u_{i,2})$.
Let $\Lambda_{i,2}\subseteq \{2,\ldots,d_{s+i}\}$ denote the
subset of indices $l\geq 2$ such that $a^{i,l}_{i,2}$ is invertible, with
term of order zero $\lambda^{i,l}_{i,2}\in\mbk$,
$\lambda^{i,l}_{i,2}\neq 0$. If $\Lambda_{i,2}$ were empty, on taking
$\sigma$-leading forms in the equality above one would deduce that
$u_{i,2}^{\sigma}=\lambda_{i,2}^{i,1}u_{i,1}^{\sigma}$, a
contradiction with $\msd_{i}^{\sigma}$ being a set of linearly independent elements of
$V_{s+i}$. Thus $\Lambda_{i,2}\neq\emptyset$ and
$u_{i,2}^{\sigma}=\lambda_{i,2}^{i,1}u_{i,1}^{\sigma}+
\sum_{l\in\Lambda_{i,2}}\lambda^{i,l}_{i,2}f_{i,l}^{\sigma}$.
Renaming the elements of $\msb_{i,1}$, we can suppose that
$2\in\Lambda_{i,2}$, so $\lambda_{i,2}^{i,2}$ and $a_{i,2}^{i,2}$ are
invertible. Set $\msb_{i,2}=\{u_{i,1},u_{i,2},f_{i,3},\ldots,
f_{i,d_{s+i}}\}$. Then, $f_{i,2}^{\sigma}$ is in the vector space spanned by 
$\msb_{i,2}^{\sigma}$ and $\msb_{i,2}^{\sigma}$ is a basis of $V_{s+i}$. Furthermore,
since $a_{i,2}^{i,2}$ is invertible, $f_{i,2}$ is in the ideal generated by 
\begin{eqnarray*}
(\msd_0\cup\msb_0^{\prime})\cup\ldots\cup(\msd_{i-1}\cup\msb_{i-1}^{\prime})\cup
\msb_{i,2}\cup\msb_{i+1}\cup\ldots\cup\msb_{\xi-1}\cup\msc_{\xi-1},
\end{eqnarray*}
which is a minimal generating set for $Q$. If $\delta_i=2$, i.e., $\msd_{i}=\{u_{i,1},u_{i,2}\}$, 
take $\msb_i^{\prime}=\msb_{i,2}\setminus\msd_{i}$, which clearly satisfies the conditions $\mbl_i$. 
Otherwise, proceed recursively to substitute all the elements of $\msd_i$ in $\msb_i$. After some renaming, one gets $\msb_{i,\delta_i}:=\{u_{i,1},\ldots,u_{i,\delta_i},f_{i,\delta_{i}+1},\ldots,f_{i,d_{s+i}}\}$, where 
\begin{eqnarray*}
(\msd_0\cup\msb_0^{\prime})\cup\ldots\cup(\msd_{i-1}\cup\msb_{i-1}^{\prime})\cup
\msb_{i,\delta_i}\cup\msb_{i+1}\cup\ldots\cup\msb_{\xi-1}\cup\msc_{\xi-1}
\end{eqnarray*}
is a minimal generating set of $Q$. Then, the set 
$\msb_i^{\prime}:=\msb_{i,\delta_i}\setminus\msd_i$ satisfies the conditions of $\mbl_i$, which finishes the whole proof. 
\end{proof}

\section{The second prime of Moh, in any characteristic}\label{section-MohN3}

The main purpose of this section is to find, 
for each characteristic of the field $\mbk$, a minimal generating set for the prime ideal $P_3$ of Moh, concretely,  when $\lambda=25$. We deduce that $\mu(P_3)$ may decrease depending on $\charac(\mbk)$, contradicting a statement of Sally. The study of the prime ideal $P_1$ of Moh is treated in Proposition~\ref{prop-casen1}, at the end of the section.

Since this part is mainly focused in $P_3$, we will omit the subscript $3$ wherever possible. Thus, $\mbk$ is a field and $\rho:=\rho_3$ is the $\mbk$-algebra morphism $\rho:\mbk\lser x,y,z\rser\to\mbk\lser t\rser$ defined by Moh in \cite{moh1}, with $n=3$, $m=2$ and $\lambda=25$, so $\rho(x)=t^6+t^{31}=t^6(1+t^{25})$, $\rho(y)=t^8$ and $\rho(z)=t^{10}$. Then, $P:=P_3=\ker(\rho)$. 
Moreover, $\sigma:\mbk\lser x,y,z\rser\to\mbk\lser x,y,z\rser$ is the $\mbk$-algebra morphism defined by $\sigma(x)=x^3$, $\sigma(y)=y^4$ and $\sigma(z)=z^5$ and $\nums=\langle 3,4,5\rangle$ is the numerical semigroup generated by $3,4,5$. Fixed $\sigma$ and $P=P_3$, and given $r\in\nums$, we consider the corresponding subspaces $W_r=W_r(\sigma)$ and $V_r=V_r(\sigma,P)$. 

\begin{remark}
For two natural numbers $m,q$, let $b_{m,q}\in\mbk$ be the image of the binomial coefficient $\binom{m}{q}\in\mbn$, where $\binom{m}{q}=0$ if $m<q$, through the ring homomorphism $\mbz\to\mbk$. 
The well-known Theorem of Lucas says that $\binom{m}{q}\equiv \prod_{i=0}^{k}\binom{m_i}{q_i} \mod p$, where
$m=m_0+m_1\cdot p+\ldots +m_k\cdot p^k$ and $q=q_0+q_1\cdot p+\ldots +q_k\cdot p^k$, $0\leq m_i,q_i<p$. In particular, if $\charac(\mbk)=p$, then $b_{m,q}=0$ if and only if $m_i<q_i$, for some $i$.
\end{remark}

\begin{lemma}\label{lemmaVr0}
Let $r\in\nums$, $r\neq 0$. Let 
$\nu_r=\max\{\alpha_1\mid\alpha\in\fac(r,\nums)\}$.
The following hold.
\begin{itemize}
\item[$(1)$]
If $\alpha\in\fac(r,\nums)$, then 
  $\rho(\mon^\alpha)=t^{2r}(1+t^{25})^{\alpha_1}=t^{2r}
  \sum_{k=0}^{\alpha_1}b_{\alpha_1,k}t^{25k}$. In particular, 
  $\rho(\mon^\alpha)\neq 0$, $\{2r,2r+25\alpha_1\}\subseteq\supp(\rho(\mon^\alpha))\subseteq \{2r+25k\mid 0\leq k\leq \alpha_1\}$ and $\ord(\rho(\mon^{\alpha}))=2r$. 
\item[$(2)$] Let $g=\sum_{\alpha\in{\rm F}(r,\mathcal{S})}\lambda_{\alpha}\mon^{\alpha}\in W_r$. 
    If $g\in V_r$, then $\sum_{\alpha\in{\rm F}(r,\mathcal{S})}\lambda_{\alpha}=0$ and
$\sum_{\alpha\in{\rm F}(r,\mathcal{S})}b_{\alpha_{1},1}\lambda_{\alpha}=0$.
\end{itemize}
Furthermore, 
\begin{itemize}
\item[$(3)$] If $\dim W_r=1$, then $V_r=\{0\}$.
\item[$(4)$] If $W_r=\langle \mon^{\alpha},\mon^{\beta}\rangle$, $\alpha\neq\beta$,
  then $V_r\subseteq\langle \mon^{\alpha}-\mon^{\beta}\rangle$.
  Moreover, if $b_{\alpha_1,1}\neq b_{\beta_1,1}$, then $V_r=\{0\}$.
\item[$(5)$] Suppose that $W_r=\langle \mon^{\alpha},\mon^{\beta},\mon^{\gamma}\rangle$, with 
$\alpha$, $\beta$, $\gamma$ distinct, and 
\begin{eqnarray*}
(b_{\gamma_1,1}-b_{\beta_1,1},b_{\alpha_1,1}-b_{\gamma_1,1},b_{\beta_1,1}-b_{\alpha_1,1})\neq (0,0,0).
\end{eqnarray*}
Then $V_r\subseteq \langle (b_{\gamma_1,1}-b_{\beta_1,1})\mon^{\alpha}+
(b_{\alpha_1,1}-b_{\gamma_1,1})\mon^{\beta}+
(b_{\beta_1,1}-b_{\alpha_1,1})\mon^{\gamma}\rangle$.
\end{itemize}
\end{lemma}
\begin{proof}
If $r\in\nums$ and $\alpha\in\fac(r,\nums)$, then
$r=\alpha_1\cdot 3+\alpha_2\cdot 4+\alpha_3\cdot 5$. Thus,
\begin{eqnarray*}
\rho(\mon^\alpha)=(t^6+t^{31})^{\alpha_1}t^{8\alpha_2}t^{10\alpha_3}=t^{2r}(1+t^{25})^{\alpha_1}=
t^{2r}\left(b_{\alpha_1,0}+b_{\alpha_1,1}t^{25}+\cdots +b_{\alpha_1,\alpha_1}t^{25\alpha_1}\right).
\end{eqnarray*}
Since $b_{\alpha_1,0}=b_{\alpha_1,\alpha_1}=1$ and $b_{\alpha_1,k}\geq 0$, for $k=1,\ldots,\alpha_1-1$, it follows that $\rho(\mon^{\alpha})\neq 0$ and 
\begin{eqnarray*}
\{2r,2r+25\alpha_1\}\subseteq\supp(\rho(\mon^\alpha))\subseteq 
\{2r+25k\mid 0\leq k\leq \alpha_1\}. 
\end{eqnarray*}
So, $\ord(\rho(\mon^\alpha))=2r$. This proves Item $(1)$.
Note that, depending on the characteristic of $\mbk$, some $b_{\alpha_1,k}$ might be zero. Therefore, 
the second containment above could be strict.

Let $g=\sum_{\alpha\in{\rm F}(r,\mathcal{S})}\lambda_{\alpha}\mon^{\alpha}\in W_r$. 
By Item $(1)$, and using that, for any $\alpha\in\fac(r,\nums)$, $\alpha_1\leq \nu_r$, then
\begin{eqnarray*}
&&\rho(g)=\sum_{\alpha\in{\rm F}(r,\mathcal{S})}
\lambda_{\alpha}t^{2r}
\left(b_{\alpha_1,0}+b_{\alpha_1,1}t^{25}+\cdots +b_{\alpha_1,\alpha_1}t^{25\alpha_1}\right)=\\
&&\sum_{\alpha\in{\rm F}(r,\mathcal{S})}
\lambda_{\alpha}t^{2r}
\left(b_{\alpha_1,0}+b_{\alpha_1,1}t^{25}+\cdots +b_{\alpha_1,\alpha_1}t^{25\alpha_1}+
\cdots+b_{\alpha_1,\nu_r}t^{25\nu_r}\right)=\\
&&\left(\sum_{\alpha\in{\rm F}(r,\mathcal{S})}b_{\alpha_1,0}\lambda_{\alpha}\right)t^{2r}+
\left(\sum_{\alpha\in{\rm F}(r,\mathcal{S})}b_{\alpha_1,1}\lambda_{\alpha}\right)t^{2r+25}+\cdots+
\left(\sum_{\alpha\in{\rm F}(r,\mathcal{S})}b_{\alpha_1,\nu_r}\lambda_{\alpha}\right)t^{2r+25\nu_r},
\end{eqnarray*}
where $b_{\alpha_1,k}=0$, whenever $k>\alpha_1$.
If $g\in\ker(\rho)$, then $\sum_{\alpha\in{\rm F}(r,\mathcal{S})}b_{\alpha_1,k}\lambda_{\alpha}=0$,
for all $0\leq k\leq \nu_r$. 
Suppose that $g\in V_r$, but $g\not\in\ker(\rho)$. 
By Lemma~\ref{lemma-Vr}, there exists $h\in\mbk\lser x,y,z\rser$, $h\neq 0$, $\sord(h)>r$, such that
$\rho(g)+\rho(h)=0$. Since $g\not\in\ker(\rho)$, then $\rho(h)\neq 0$.
Write $h=\sum_{s>r}h_{(s)}$, where
$h_{(s)}=\sum_{\alpha\in {\rm F}(s,\mathcal{S})}\mu_{\alpha}\mon^{\alpha}\in W_s$.
As before, and letting 
$\nu_s=\max\{\alpha_1\mid\alpha\in\fac(s,\nums)\}$, then 
\begin{eqnarray}\label{eq-rhoh}
&&\rho(h)=\sum_{s>r}\rho(h_{(s)})=\\
&&\sum_{s>r}\left[\left(\sum_{\alpha\in{\rm F}(s,\mathcal{S})}b_{\alpha_1,0}\mu_{\alpha}\right)t^{2s}+
\left(\sum_{\alpha\in{\rm F}(s,\mathcal{S})}b_{\alpha_1,1}\mu_{\alpha}\right)t^{2s+25}+\cdots+
\left(\sum_{\alpha\in{\rm F}(s,\mathcal{S})}b_{\alpha_1,\nu_s}\mu_{\alpha}\right)t^{2s+25\nu_s}\right].
\nonumber
\end{eqnarray}
Note that $t^{2s+25k}=t^{2s'+25k'}$ could occur for some $s'>s$ and some integers $k>k'$. But, if $s>r$, 
then $2r<2s<2s+25k$, $2r+25\neq 2s$ and $2r+25<2s+25k$, for every $k\geq 1$. Thus, for any $s>r$,
$t^{2s+25k}$ cannot be equal to $t^{2r}$ or $t^{2r+25}$. From the equality $\rho(g)+\rho(h)=0$
one deduces that
$\sum_{\alpha\in{\rm F}(r,\mathcal{S})}b_{\alpha_1,0}\lambda_{\alpha}=0$ and 
$\sum_{\alpha\in{\rm F}(r,\mathcal{S})}b_{\alpha_1,1}\lambda_{\alpha}=0$.
This proves Item $(2)$.

If $W_r=\langle \mon^\alpha\rangle$ and $g=\lambda\mon^{\alpha}\in V_r$, then, by Item $(2)$, 
$\lambda=0$ and $g=0$. This proves Item $(3)$.

Suppose that $W_r=\langle \mon^{\alpha},\mon^{\beta}\rangle$ and $g=\lambda\mon^{\alpha}+\mu\mon^{\beta}\in W_r$. 
By $(2)$, $\lambda+\mu=0$ and $b_{\alpha_1,1}\lambda+b_{\beta_1,1}\mu=0$. Thus, 
$g=\lambda(\mon^{\alpha}-\mon^{\beta})$. Moreover, 
if $b_{\alpha_1,1}\neq b_{\beta_1,1}$, then $\lambda=0$ and $g=0$, which proves Item $(4)$.

Suppose that $W_r=\langle \mon^\alpha,\mon^{\beta},\mon^{\gamma}\rangle$ and let 
$g=\lambda\mon^{\alpha}+\mu\mon^{\beta}+\eta\mon^{\gamma}\in V_r$. By Item $(2)$, 
$\lambda+\mu+\eta=0$ and $b_{\alpha_1,1}\lambda+b_{\beta_1,1}\mu+b_{\gamma_1,1}\eta=0$. 
By hypothesis, 
\begin{eqnarray*}
\rank\left(\begin{array}{ccc}1&1&1\\b_{\alpha_1,1}&b_{\beta_1,1}&b_{\gamma_1,1}\end{array}\right)=2
\end{eqnarray*}
and $(\lambda,\mu,\eta)\in\langle (b_{\gamma_1,1}-b_{\beta_1,1},b_{\alpha_1,1}-b_{\gamma_1,1},b_{\beta_1,1}-b_{\alpha_1,1})\rangle$, which shows Item $(5)$.
\end{proof}

\begin{convention}\label{grevlex}
From now on, a subset of monomials will be ordered from biggest to smallest, according to the local {\em negative degree reverse lexicographical order} $>_{\rm ds}$ on the set of monomials $\monset_d=\{\mon^{\alpha}\mid \alpha\in\mbn^d\}$, that is: $\mon^\alpha >_{\rm ds} \mon^\beta \Leftrightarrow \deg(\mon^\alpha)<\deg(\mon^\beta)$ or $\deg(\mon^\alpha)=\deg(\mon^\beta)$ and there exists $1\leq i\leq d-1$ such that $\alpha_d=\beta_d,\ldots, \alpha_{i+1} = \beta_{i+1}$ and $\alpha_i<\beta_i$. For instance, and for $d=3$, 
\begin{eqnarray*}
1>x>y>z>x^2>xy>y^2>xz>yz>z^2>x^3>x^2y>xy^2>y^3>x^2z>xyz>\ldots .
\end{eqnarray*}
A non-zero polynomial $f$ will be written as $f=\sum_{\nu=0}^{n}a_{\nu}\mon^{\alpha(\nu)}$, $a_\nu\in\mbk\setminus\{0\}$, with $\mon^{\alpha(0)}>\ldots>\mon^{\alpha(n)}$, so the leading monomial of $f$ is $\lmon(f)=\mon^{\alpha(0)}$ (see Notation~\ref{notation-ord}). 
A finite subset of polynomials $f_1,\ldots,f_s$ with distinct leading monomial terms will be listed according to $\lmon(f_1)>\ldots>\lmon(f_s)$. 
We will usually underline the $\sigma$-leading form of any polynomial with more than one term.
\end{convention}

\begin{remark}\label{remark-W16}
A simple calculation of $\fac(r,\nums)$ gives rise to the following $W_r$, $r\in\nums$, $3\leq r\leq 16$:
\begin{eqnarray*}
&&W_{3}=\langle x\rangle,\; W_{4}= \langle y\rangle,\; W_{5}= \langle z\rangle,\; W_{6}= \langle x^2\rangle,\;
W_{7}= \langle xy\rangle,\; W_{8}= \langle y^2, xz\rangle,\;
W_9=\langle yz,x^3\rangle,\; \\
&&W_{10}=\langle z^2,x^2y\rangle,\; 
W_{11}=\langle xy^2,x^2z\rangle,\;
W_{12}=\langle y^3,xyz,x^4\rangle,\;
W_{13}=\langle y^2z,xz^2,x^3y\rangle,\; \\
&&W_{14}=\langle yz^2,x^2y^2,x^3z\rangle,\; 
W_{15}= \langle z^3,xy^3,x^2yz,x^5\rangle,\;
W_{16}=\langle x^4y,y^4,xy^2z,x^2z^2\rangle.
\end{eqnarray*}
\end{remark}

\begin{lemma}\label{lemma-s}
Let $s=\min\{\sord(f)\mid f\in P,\; f\neq 0\}$. If $\charac(\mbk)=0$ or $p\geq 5$, then $s=12$.
If $\charac(\mbk)=2$, then $s=10$. If $\charac(\mbk)=3$, then $s=9$.
\end{lemma}
 \begin{proof}
By Lemma~\ref{lemmaVr0}, $V_r=\{0\}$, 
for $3\leq r\leq 7$. Suppose that $\charac(\mbk)=0$ or $p\geq 5$.
By Lemma~\ref{lemmaVr0}, we deduce that:
\begin{itemize}
\item $W_8=\langle \mon^\alpha,\mon^\beta\rangle$, with $\alpha=(0,2,0)$, $\beta=(1,0,1)$,
$b_{\alpha_1,1}-b_{\beta_1,1}=-1$ and $V_8=\{0\}$;
\item $W_9=\langle \mon^\alpha,\mon^\beta\rangle$, with $\alpha=(0,1,1)$, $\beta=(3,0,0)$,
$b_{\alpha_1,1}-b_{\beta_1,1}=-3$ and $V_9=\{0\}$;
\item $W_{10}=\langle \mon^\alpha,\mon^\beta\rangle$, with $\alpha=(0,0,2)$, $\beta=(2,1,0)$,
$b_{\alpha_1,1}-b_{\beta_1,1}=-2$ and $V_{10}=\{0\}$;
\item $W_{11}=\langle \mon^\alpha,\mon^\beta\rangle$, with $\alpha=(1,2,0)$, $\beta=(2,0,1)$,
$b_{\alpha_1,1}-b_{\beta_1,1}=-1$ and $V_{11}=\{0\}$;
\item $W_{12}=\langle \mon^\alpha,\mon^\beta,\mon^{\gamma}\rangle$, with $\alpha=(0,3,0)$, $\beta=(1,1,1)$, 
$\gamma=(4,0,0)$ and \\
$(b_{\gamma_1,1}-b_{\beta_1,1},b_{\alpha_1,1}-b_{\gamma_1,1},b_{\beta_1,1}-b_{\alpha_1,1})=(3,-4,1)$.
\end{itemize}
Let us find $f_1\in P$, $f_1\neq 0$, with $f_1^{\sigma}\in V_{12}$. By Lemma~\ref{lemmaVr0},
$V_{12}\subseteq \langle 3y^3-4xyz+x^4\rangle$, so pick $f_1^{\sigma}=3y^3-4xyz+x^4$. Then,  $\rho(f_1^{\sigma})=6t^{74}+4t^{99}+t^{124}$. Now we want to construct the tail $f_1^{\tau}$.  As a general rule, we will try to find a $\sigma$-homogeneous tail whose monomials have exponents in $x$ as smallest as possible. 
By Lemma~\ref{lemmaVr0}, if $r\in\nums$, $r\neq 0$, and $\alpha\in\fac(r,\nums)$, then $\ord(\rho(\mon^\alpha))=2r$. Thus, pick $f_1^{\tau}\in W_{37}$, where $W_{37}$ is generated by the monomials
\begin{eqnarray*}
y^3z^5,xyz^6,y^8z,xy^6z^2,x^2y^4z^3,x^3y^2z^4,x^4z^5,x^3y^7,x^4y^5z,x^5y^3z^2,
x^6yz^3,x^7y^4,x^8y^2z,x^9z^2,x^{11}y.
\end{eqnarray*}
Note that the image by $\rho$ of the monomials  $y^{\alpha_2}z^{\alpha_3}\in W_{37}$ is just $t^{74}$. However, we need to get a $t^{99}$, which has an odd exponent. Thus, $f_1^{\tau}$ must contain at least a monomial multiple of $x$. On the other hand, to get the term $t^{124}$, we need a monomial multiple of $x^2$. Considering that $\rho(xy^6z^2)=t^{74}+t^{99}$ and
that $\rho(x^2y^4z^3)=t^{74}+2t^{99}+t^{124}$, we see that one can take $f_1^{\tau}=-3y^3z^5-2xy^6z^2-x^2y^4z^3$. Therefore, let
$f_1:=\underline{3y^3-4xyz+x^4}-3y^3z^5-2xy^6z^2-x^2y^4z^3$. Then, $f^{\sigma}_1=3y^3-4xyz+x^4\in W_{12}$, 
$f_1\in \ker(\rho)=P$, $f_1^{\sigma}\in V_{12}$ and $s=\min\{\sord(f)\mid f\in P,\; f\neq 0\}=12$. 

Suppose that $\charac(\mbk)=2$. From the previous discussion we deduce that
$V_r=\{0\}$, for $3\leq r\leq 9$. Let us find $g_1\in P$, $g_1\neq 0$, with $g_1^{\sigma}\in V_{10}$. 
As before, by Lemma~\ref{lemmaVr0}, $g_1^{\sigma}$ must be of the form $g_1^{\sigma}=z^2+x^2y$. Then, 
$\rho(g_1^{\sigma})=t^{70}$. Take the monomial $y^5z^3\in W_{35}$, which has no $x$ and whose image is 
$\rho(y^5z^3)=t^{70}$. Let $g_1:=\underline{z^2+x^2y}+y^5z^3$. Then, 
$g_1^{\sigma}=z^2+x^2y\in W_{10}$, $g_1 \in\ker(\rho)=P$, $g_1^{\sigma}\in V_{10}$ and 
$s=\min\{\sord(f)\mid f\in P,\; f\neq 0\}=10$.

Suppose that $\charac(\mbk)=3$. As before, we can affirm that $V_r=\{0\}$, for $3\leq r\leq 8$. 
Let us find $h_1\in P$, $h_1\neq 0$, with $h_1^{\sigma}\in V_{9}$. By Lemma~\ref{lemmaVr0}, 
$h_1^{\sigma}=yz-x^3$. Then, $\rho(h_1^{\sigma})=-t^{93}$. To vanish this term, the tail must contain a
monomial multiple of $x$, so we pick one of the form $xy^{\alpha_2}z^{\alpha_3}$. Then, $\rho(xy^{\alpha_2}z^{\alpha_3})=t^{6+8\alpha_2+10\alpha_3}+t^{31+8\alpha_2+10\alpha_3}$. 
A possible solution of $31+8\alpha_2+10\alpha_3=93$ is $(\alpha_2,\alpha_3)=(4,3)$. 
Note that $xy^4z^3\in W_{34}$ and
$\rho(xy^4z^3)=t^{68}+t^{93}$. To vanish $t^{68}$, we can choose $y^6z^2$, which is in $W_{34}$. So, 
let $h_1=\underline{yz-x^3}-y^6z^2+xy^4z^3\in\mbk\lser x,y,z\rser$. Then, 
$h_1^{\sigma}=yz-x^3\in W_{9}$, $h_1\in \ker(\rho)=P$, $h_1^{\sigma}\in V_{9}$ and 
$s=\min\{\sord(f)\mid f\in P,\; f\neq 0\}=9$.
\end{proof}

\begin{discussion}\label{discussion-0}
Suppose that $\charac(\mbk)=0$ or $p>5$. Let us find $f_1,f_2,f_3,f_4\in P$, $f_i\neq 0$, such that
$f_1^{\sigma}\in V_{12}$, $f_2^{\sigma}\in V_{13}$, $f_3^{\sigma}\in V_{14}$ and 
$f_4^{\sigma}\in V_{15}$. Observe that $s=\min\{\sord(f)\mid f\in P,\; f\neq 0\}=12$ and $\xi=3$. 
So, by Theorem~\ref{moh43-2}, we are interested in finding  elements whose $\sigma$-leading form are 
in $V_{12}$, $V_{13}$ and $V_{14}$. Moreover, by \cite[Theorem]{moh2}, we also have to consider $V_{15}$, 
where, for $n=3$, $n^2+2n=15$.

By Lemma~\ref{lemma-s}, we can pick $f_1=\underline{3y^3-4xyz+x^4}-3y^3z^5-2xy^6z^2-x^2y^4z^3$. 

By Lemma~\ref{lemmaVr0}, we can pick $f_2^{\sigma}=2y^2-3xz^2+x^3y\in V_{13}$. Then,
$\rho(f_2^{\sigma})=3t^{76}+t^{101}$. Let us find $f_2^{\tau}\in W_{38}$ with
monomials whose exponents in $x$ be as smallest as possible. To vanish $t^{101}$, take
$xy^{\alpha_2}z^{\alpha_3}$, so $\rho(xy^{\alpha_2}z^{\alpha_3})=t^{6+8\alpha_2+10\alpha_3}+t^{31+8\alpha_2+10\alpha_3}$. To obtain $31+8\alpha_2+10\alpha_3=101$, choose $(\alpha_2,\alpha_3)=(5,3)$. So,
$xy^5z^3\in W_{38}$ and $\rho(xy^5z^3)=t^{76}+t^{101}$. To vanish $t^{76}$, choose $y^7z^2\in W_{38}$. So, we can pick the polynomial $f_2=\underline{2y^2-3xz^2+x^3y}-2y^7z^2-xy^5z^3$. 

By Lemma~\ref{lemmaVr0}, we can pick $f_3^{\sigma}=yz^2-3x^2y^2+2x^3z\in V_{14}$. Then, $\rho(f_3^{\sigma})=3t^{78}+2t^{103}$. Let us find $f_3^{\tau}\in W_{39}$ with
monomials whose exponents in $x$ be as smallest as possible. To vanish $t^{103}$, take
$xy^{\alpha_2}z^{\alpha_3}$ and solve $31+8\alpha_2+10\alpha_3=103$. Thus, we can choose $(\alpha_2,\alpha_3)=(4,4)$. So, $xy^4z^4\in W_{39}$ and $\rho(xy^4z^4)=t^{78}+t^{103}$. To vanish $t^{78}$, choose $y^6z^3\in W_{39}$. 
So, we can pick the polynomial $f_3=\underline{yz^2-3x^2y^2+2x^3z}-y^6z^3-2xy^4z^4$. 

Finally, take $f_4^{\sigma}=az^3+bxy^3+cx^2yz\in V_{15}$, omitting
the $x^5$ for the sake of simplicity. Then,
\begin{eqnarray*}
\rho(f_4^{\sigma})=(a+b+c)t^{30}+(b+2c)t^{55}+ct^{80}.    
\end{eqnarray*}
Choose $(a,b,c)=(1,-2,1)$, so $f_4^{\sigma}=z^3-2xy^3+x^2yz$ and $\rho(f_4^{\sigma})=t^{80}$. To vanish $t^{80}$, choose $y^5z^4\in W_{40}$. Thus, we can take $f_4=\underline{z^3-2xy^3+x^2yz}-y^5z^4$. 

Note that the $\sigma$-leading forms $f_1^{\sigma}$, $f_2^{\sigma}$, $f_3^{\sigma}$, $f_4^{\sigma}$, coincide with the homonymous elements found in \cite[Exemple~4.4]{mss}, though in that example $\lambda=27$.
\end{discussion}

In the next theorem we prove that the polynomials $f_1,f_2,f_3,f_4$ form a minimal generating set of $P$ if $\charac(\mbk)=0$ or $p\geq 5$. The demonstration consists in proving that the ideal $I:=(f_1,f_2,f_3,f_4)$ is equal to $P$, where, by construction $I\subseteq P$. To prove the other inclusion we calculate the lengths of 
$R/(I+yR)$ and $R/(P+yR)$, where $R=\mbk\lser x,y,z\rser$, and see that they coincide. The advantage of this method is that it can be reproduced whatever the characteristic of $\mbk$ is. To not interfere with the development of the demonstration, we recall in Remark~\ref{rem-theoretical} some definitions and theoretical results.

\begin{theorem}\label{theo:char0}
If $\charac(\mbk)=0$ or $p\geq 5$, then $P$ is minimally generated by 
\begin{eqnarray*}
&&f_1=\underline{3y^3-4xyz+x^4}-3y^3z^5-2xy^6z^2-x^2y^4z^3,\phantom{+}
f_2=\underline{2y^2z-3xz^2+x^3y}-2y^7z^2-xy^5z^3,\\
&&f_3=\underline{yz^2-3x^2y^2+2x^3z}-y^6z^3-2xy^4z^4,\phantom{+}
f_4=\underline{z^3-2xy^3+x^2yz}-y^5z^4.
\end{eqnarray*}
\end{theorem}
\begin{proof}
Let $R=\mbk\lser x,y,z\rser$ and $\mfm=(x,y,z)$.
Observe that $R$ is a regular local ring with maximal ideal $\mfm$ 
(see, e.g., \cite[After Definition~2.2.1]{bh}).
In particular, $R$ is a Cohen-Macaulay local ring (see, e.g., \cite[Definitions~2.2.1 and 2.1.1, and Corollary~2.2.6]{bh}). Set $I=(f_1,f_2,f_3,f_4)$. 
By Discussion~\ref{discussion-0}, $I\subseteq P$. For the sake of easy legibility, 
we divide the proof in several steps.

\vspace*{0,3cm}

\noindent \underline{\sc Step $1$}: If $I=P$, then $\mu(P)=4$ and $f_1,f_2,f_3,f_4$ is a minimal generating set of $P$.

\vspace*{0,3cm}

Take in Theorem~\ref{moh43-2}, $Q=P$ and $\sigma:\mbk\lser\monx\rser\rightarrow\mbk\lser\monx\rser$ 
defined by $\sigma(x)=x^3$, $\sigma(y)=y^4$, $\sigma(z)=z^5$, so $\xi=\min(n_1,n_2,n_3)=3$. 
By Lemma~\ref{lemma-s}, $s=\min\{\sord(f)\mid f\in Q, f\neq 0\}=12$. By Discussion~\ref{discussion-0}, $f_1,f_2,f_3,f_4\in P$ and
$f_1^{\sigma}\in V_{12}$, $f_2^{\sigma}\in V_{13}$, $f_3^{\sigma}\in V_{14}$ and $f_4^{\sigma}\in V_{15}$.
By Theorem \ref{moh43-2}, it follows that $f_1,f_2,f_3$ is part of a minimal generating set of $P$, so 
$\mu(P)\geq 3$. Since $(f_1,f_2,f_3)\subset (x,y)$ and $f_4=z^3+g$, with $g\in (x,y)$, it follows that $f_4\not\in (f_1,f_2,f_3)$. If $\mu(P)=3$, then $(f_1,f_2,f_3)=P$ and, 
since $f_4\in P$, then $f_4\in (f_1,f_2,f_3)$, a contradiction. Therefore, $\mu(P)>3$.
But, if $I=P$, then $\mu(P)\leq 4$. So, $\mu(P)=4$ and $f_1,f_2,f_3,f_4$ is a minimal generating set of $P$.

\vspace*{0,3cm}

\noindent \underline{\sc Step $2$}: $\height(I),\height(P)=2$ and $\length_R(R/(I+yR)),\length_R(R/(P+yR))$ are finite. 

\vspace*{0,3cm}

Clearly, $(f_1,f_4,y)=(x^4,z^3,y)$ and $x^4,z^3,y$ is a regular sequence. 
By \cite[Corollary~1.6.19]{bh}, we get
$\grade(f_1,f_4,y)=\grade(x^4,z^3,y)=3$ and $f_1,f_4,y$ is a regular sequence.
In particular, since $P\subsetneq \mfm$, then $\height(P)<\height(\mfm)=3$. So,
$2\leq\grade(I)=\height(I)\leq \height(P)\leq 2$ and 
$\height(I)=\height(P)=2$.
Since $(f_1,f_4,y)\subset I+yR\subseteq P+yR$, it follows that
$3=\grade(f_1,f_4,y)\leq\grade(I+yR)\leq\grade(P+yR)\leq \dim R=3$. Therefore, 
$\height(I+yR)=3$, $\height(P+yR)=3$ and $R/(I+yR)$, $R/(P+yR)$ are Artinian rings. Thus, 
$R/(I+yR)$ and $R/(P+yR)$ have finite length as $R$-modules. 

\vspace*{0,3cm}

\noindent \underline{\sc Step $3$}: $\length_R(R/(I+yR))=8$.

\vspace*{0,3cm}

Set $J=I+yR=(x^4,xz^2,x^3z,z^3,y)$. Let $S=R/yR$, which is a local ring with maximal ideal $\mfm_S=\mfm/yR$, say, with residue field $S/\mfm_S=\mbk$. By abuse of notation, we still denote $x,z$ to the classes of the variables in $S$. Let $\overline{J}=J/yR=(x^4,xz^2,x^3z,z^3)$. Since $R\to S$ is a surjective map and $S/\overline{J}=R/J$ is an $S$-module, it follows
that $\length_R(R/(I+yR))=\length_R(R/J)=\length_R(S/\overline{J})=\length_S(S/\overline{J})$. 
Observe that $\overline{J}\subset \mfm_S^3$ and consider the short exact sequences: 
\begin{multline*}
0\to \mfm_S/\overline{J}\to S/\overline{J}
\to S/\mfm_S\to 0,\;\\
0\to \mfm_S^{2}/\overline{J}\to\mfm_S^{}/\overline{J}
\to\mfm_S/\mfm_S^2\to 0,\;\\
0\to \mfm_S^{3}/\overline{J}\to\mfm_S^2/\overline{J}
\to\mfm_S^2/\mfm_S^3\to 0.
\end{multline*}
Using the additivity of length and that the length of a vector space coincides with its dimension
(see, e.g., \cite[Propositions~6.9 and 6.10]{am}), 
\begin{multline*}
\length_{S}(S/\overline{J})=\length_S(\mfm_S/\overline{J})+\length_S(S/\mfm_S)=
\length_S(\mfm_S^{2}/\overline{J})+\length_S(\mfm_S/\mfm_S^2)+1=\\
\length_S(\mfm_S^{3}/\overline{J})+\length_S(\mfm_S^2/\mfm_S^3)+2+1=
\length_S(\mfm_S^3/\overline{J})+6. 
\end{multline*}

Since $\mfm_S(\mfm_S^3/\overline{J})=0$, it follows that $\mfm_S^{3}/\overline{J}=\langle x^3,x^2z\rangle$ is a two-dimensional $S/\mfm_S$-vectorial space, so $\length_S(\mfm_S^{3}/\overline{J})=\length_{\Bbbk}(\mfm_S^{3}/\overline{J})=\dim_{\Bbbk}(\mfm_S^{3}/\overline{J})=2$. Thus, $\length_S(S/\overline{J})=8$.

\vspace*{0,3cm}

\noindent \underline{\sc Step $4$}: $D:=R/P=\mbk\lser t^6+t^{31},t^8,t^{10}\rser$ is a one-dimensional 
Cohen-Macaulay local domain. 

\vspace*{0,3cm}

Since $R$ is Cohen-Macaulay and local, $\dim(D)=\dim(R/P)=\dim(R)-\height(P)=1$ (see, e.g., \cite[Proposition~2.1.4]{bh}). Since $R$ is Noetherian local and $P$ is a prime ideal of $R$, then $D=R/P$ is a Noetherian local domain. In particular, any non-zero and non-invertible element of $D$ 
is a $D$-regular sequence, so  
$1\leq\grade(\mfm_D,D)=\depth(D)\leq\dim(D)=1$ and $D$ is a Cohen-Macaulay ring
(see, e.g., \cite[Definition~1.2.7, Proposition~1.2.12 and Definition~2.1.1]{bh}).
Let $\mfm_D=\mfm/P$ be its maximal ideal and $k_D=D/\mfm_D=R/\mfm=\mbk$ its residue field. By abuse of notation, let $x,y,z$ still denote the classes of the variables in $D$, so $\mfm_D=(x,y,z)$. 
Note that $D=R/P=\mbk\lser x,y,z\rser/\ker(\rho)\cong \im(\rho)=\mbk\lser t^6+t^{31},t^8,t^{10}\rser$. 
Through this isomorphism, $x$, $y$ and $z$ are identified with $t^6+t^{31}$, $t^8$ and $t^{10}$, respectively. 

\vspace*{0,3cm}

\noindent \underline{\sc Step $5$}: $D[t]=\mbk\lser t\rser$.

\vspace*{0,3cm}

Clearly, $D[t]=\mbk\lser t^6+t^{31},t^8,t^{10}\rser[t]$ is a subring of $\mbk\lser t\rser$. Let us prove the equality. To do that, consider the numerical semigroup $S=\langle 4,5\rangle$, whose Frobenius number is $F(S)=11$ (see, e.g., \cite[Theorem~2.1.1]{ramirez}).
Let $M$ be the subsemigroup $M=2S=\{2s\mid s\in S\}=\langle 8,10\rangle$. Clearly, 
the last even element not in $M$ is $2F(S)=22$. 

Let $f\in\mbk\lser t\rser$. Write $f=\varphi_1+\varphi_2+\varphi_3$, where
$\varphi_1:=\sum_{i=0}^{54}\lambda_it^{i}$, $\varphi_2:=\sum_{i\geq 0}\lambda_{55+2i}t^{55+2i}$ and 
$\varphi_3:=\sum_{i\geq 0}\lambda_{56+2i}t^{56+2i}$. Clearly $\varphi_1\in\mbk[t]\subset D[t]$. 
Note that $\varphi_2=\sum_{i\geq0}\lambda_{55+2i}((t^6+t^{31})t^{24+2i}-t^{30+2i})$. 
Since $22$ is the last even number not in $\langle 8,10\rangle$, then, for each $i\geq 0$, 
both $t^{24+2i}$ and $t^{30+2i}$ are in $\mbk\lser t^8,t^{10}\rser\subset D$. Thus, 
$(t^6+t^{31})t^{24+2i}-t^{30+2i}\in\mbk\lser t^6+t^{31},t^8,t^{10}\rser$ and 
$\varphi_2\in D$. Again, since $22$ is the last even number not in $\langle 8,10\rangle$, then, for each $i\geq 0$, $t^{56+2i}\in \mbk\lser t^8,t^{10}\rser\subset D$. Therefore, $\varphi_3\in D$. So, $f\in D[t]$ and $D[t]=\mbk\lser t\rser$. 

\vspace*{0,3cm}

\noindent \underline{\sc Step $6$:} $K(D)=K(\mbk\lser t\rser)$, 
where $K(B)$ denotes the field of fractions of a domain $B$. 

\vspace*{0,3cm}

Since $D=\mbk\lser t^6+t^{31},t^8,t^{10}\rser\subset
\mbk\lser t\rser$, then $K(D)\subseteq K(\mbk\lser t\rser)$.
To see the other inclusion, take $g\in K(\mbk\lser t\rser)$. 
By {\sc Step 5}, $\mbk\lser t\rser=D[t]$. Thus, 
$g\in K(\mbk\lser t\rser)=K(D[t])$ and
$g=p(t)/q(t)$, where $p(t),q(t)\in D[t]$, $q(t)\neq 0$.
Since $t^8$ and $t^{10}$ are in $D$ and $t^8t^{-10}=t^{-2}$, 
then $t^{-2}$, and $t^2$, are in $K(D)$. So,
$(t^6+t^{31}-(t^{2})^3)(t^{-2})^{15}=t\in K(D)$. 
Since $K(D)$ is a field that contains $t$ and $D$, it follows that $K(D)$
contains too any quotient of polynomial expressions in $t$ with coefficients in $D$, thus
$g=p(t)/q(t)\in K(D)$.

\vspace*{0,3cm}

\noindent \underline{\sc Step $7$}: The integral closure of $D=\mbk\lser t^6+t^{31},t^8,t^{10}\rser$
in its field of fractions is $\mbk\lser t\rser$.

\vspace*{0,3cm}

The element $t$ is a root of the monic polynomial $T^8-t^8\in D[T]$.
Thus, $t$ is integral over $D$ and $D\subset D[t]$ is an integral extension 
(see, e.g., \cite[Proposition~5.1, $(iii)\Rightarrow (i)$]{am}). 
Recall that $\mbk\lser t\rser$ is a regular local ring, in particular, 
it is integrally closed in its field of fractions (see, e.g., 
\cite[Corollary~2.2.20 and the definition before]{bh}). 
Since $D\subset D[t]=\mbk\lser t\rser$ is an integral extension, 
$\mbk\lser t\rser$ is integrally closed in $K((\mbk\lser t\rser)$
and $K(D)=K(\mbk\lser t\rser)$, then the integral closure of $D$ 
is $\mbk\lser t\rser$ (see Remark~\ref{rem-theoretical}, {\sc Step 7}). 

\vspace*{0,3cm}

\noindent \underline{\sc Step $8$}: If $V=\mbk\lser t\rser$, then
$\length_R(R/(P+yR))=\length_D(D/yD)=\length_D(V/yV)$. 

\vspace*{0,3cm}

Since $R\to D=R/P$ is a surjective map and $D/yD=R/(P+yR)$ is a $D$-module, it follows that
$\length_R(R/(P+yR))=\length_R(D/yD)=\length_D(D/yD)$ (see Remark~\ref{rem-theoretical}, {\sc Step 3}).

Since $y$ is a $D$-regular sequence in $\mfm_D$
and $D$ is a one-dimensional Cohen-Macaulay local ring, 
then $y$ is a system of parameters of $D$ 
(see, e.g., \cite[\S14]{matsumura} and \cite[Theorem~2.1.2]{bh}). 


Let $V=\mbk\lser t\rser$. Since $V$ is a one-dimensional regular local ring, then 
$V$ is a discrete valuation ring, DVR for short (see, e.g., \cite[Theorem~11.2]{matsumura}). 
Set $\mfm_V=(t)$ its maximal ideal and $k_V=\mbk$ its residue field. 

We have seen that $D\subset D[t]=\mbk\lser t\rser=V$, where $t$ is integral over $D$. 
Thus, $V$ is a finitely generated $D$-module (see, e.g., \cite[Proposition~5.1]{am}). 
Since $t^8$ is a $V$-regular sequence in $\mfm_D$, then
$1\leq \grade(\mfm_D,V)=\depth_D(V)\leq \dim V=1$ and $V$ is a finitely generated 
Cohen-Macaulay $D$-module 
(see, e.g., \cite[Definition~1.2.7, Proposition~1.2.12 and Definition~2.1.1]{bh}).

Note that $K(D)=(D\setminus\{0\})^{-1}D\subseteq (D\setminus\{0\})^{-1}V
\subseteq (D\setminus\{0\})^{-1}K(D)=K(D)$. Therefore, 
$V\otimes_DK(D)=(D\setminus\{0\})^{-1}V=K(D)$.
So, $\rank_D(V)=1$ (see, e.g., \cite[Definition~1.4.2]{bh}). 
Using that $D$ is a Cohen-Macaulay local ring, that $V$ is a finitely generated
Cohen-Macaulay $D$-module of $\rank_D(V)=1$ and that
$y$ is a system of parameters of $D$, we conclude that 
$\length_D(V/yV)=\length_D(D/yD)\cdot \rank_D(V)=\length_D(D/yD)$
(see, e.g., \cite[Corollary~4.6.11]{bh}).

\vspace*{0,3cm}

\noindent \underline{\sc Step $9$}: $\length_D(V/yV)=\length_V(V/yV)=8$.

\vspace*{0,3cm}

Let $0=M_8\subset M_7\subset\ldots\subset M_{1}\subset M_0=V/yV$ be the chain defined by 
$M_i=\mfm_V^{i}/\mfm_V^{8}$, $i=0,\ldots,8$, where each $M_i$ is both a $V$-module and a $D$-module. 
Note that $M_0=V/t^8V=V/yV$. Each consecutive quotient satisfies 
$M_i/M_{i+1}=\mfm_V^i/\mfm_V^{i+1}\cong k_V=k_D$, which is both a simple 
$V$-module and a simple $D$-module (see Remark~\ref{rem-theoretical}, {\sc Step 2}). 
Thus, this chain is a series of composition of $V/yV$, 
considered as a $V$-module and considered as a $D$-module. It follows that $\length_D(V/yV)=\length_V(V/yV)=8$. In particular, by {\sc Steps} 8 and 9, 
$\length_R(R/(P+yR))=8$. 

\vspace*{0,3cm}

\noindent \underline{\sc Step 10: Conclusion}. 

\vspace*{0,3cm}

Since $\rho(y)=t^8$, then $y\not\in P$, $P\cap yR=yP$ and 
$\text{Tor}_1(R/P,R/yR)=(P\cap yR)/yP=0$. Consider the short exact sequence, $0\rightarrow P/I\rightarrow R/I\rightarrow R/P\rightarrow 0$. On tensoring by $R/yR$ we obtain the exact sequence 
\begin{eqnarray*}
\cdots\rightarrow\underbrace{\text{Tor}_1(R/P,R/yR)}_{=0}\rightarrow (P/I)/y(P/I)\rightarrow R/(I+yR)\rightarrow R/(P+yR)\rightarrow 0.
\end{eqnarray*}
By {\sc Steps} 3, 8 and 9, $\length_R(R/(I+yR))=\length_R(R/(P+yR))=8$.
By the additivity of the length, we get $\length_R((P/I)/y(P/I))=0$ and $P/I=y(P/I)$. 
By Nakayama's Lemma, $I=P$.
 \end{proof}

\begin{remark}\label{rem-theoretical}
Let us recall some definitions and results appearing in the proof of Theorem~\ref{theo:char0}.

\vspace*{0,2cm}

\noindent \underline{\sc Step $1$}. 

\vspace*{0,2cm}

\noindent Here we are implicitly using that, in a Noetherian local ring, any generating set of an ideal can be reduced to a minimal generating set and that all the minimal generating sets of an ideal have the same cardinality (see, e.g., \cite[Theorem~2.3]{matsumura}).

\vspace*{0,2cm}

\noindent \underline{\sc Step $2$}. 

\vspace*{0,2cm}

\noindent Remember that the $\grade(\mfa)$ of an ideal $\mfa$ 
is the length of a maximal regular sequence in $\mfa$. One has $\grade(\mfa)\leq\height(\mfa)$ (see, e.g., \cite[Definition~1.2.6 and Proposition~1.2.14]{bh}). If
$R$ is a Noetherian Cohen-Macaulay ring and if $\mfa\neq R$, then $\grade(\mfa)=\height(\mfa)$ and 
$\height(I)+\dim R/I=\dim R$ (see, e.g., \cite[Definition~2.1.1 and Proposition~2.1.4]{bh}).
Suppose that $\height(\mfa)=\dim R$. Then $\dim R/\mfa=0$. In such a case, $R/\mfa$ is a Noetherian ring of dimension zero, thus, an Artinian ring (see, e.g., \cite[Theorem~8.5]{am}). In particular, 
$R/\mfa$ satisfies both the ascending chain condition and the descending chain condition. Thus, $R/\mfa$ has a composition series. Recall that a composition series of an $R$-module $M$ is
a chain of submodules $M=M_0\supsetneq M_1\supsetneq\ldots\supsetneq M_n=(0)$, where 
each $M_i/M_{i+1}$ is a simple $R$-module, i.e. it has no proper submodules. The number of containments $n$ is denoted $\length_R(M)$ and called the length of $M$ (see, e.g. \cite[Definition before Proposition~6.7 and Proposition~6.8]{am}). If $M$ is a simple $R$-module, then $M$ is principal, for, 
if $x\in M$, $x\neq 0$, the kernel of $f:M\to M/(x)$ contains $x$, thus $\ker(f)=M$ and $M=(x)$. In particular, if $M\neq 0$, and $(R,\mfm,k)$ is local, then $g:R\to M$, $g(1)=x$, is an epimorphism, $R/\ker(g)\cong M$ and, since $M$ is simple, $\ker(g)=\mfm$ and $M\cong R/\mfm=k$.

\vspace*{0,2cm}

\noindent \underline{\sc Step $3$}. 

\vspace*{0,2cm}

\noindent If $R\to S$ is a surjective map of rings, then $S\cong R/\mfa$, for some ideal $\mfa$ of $R$. 
Fix an $S$-module $M$. By restriction of scalars, $M$ is an $R$-module. In particular, $\mfa\cdot M=0$.
Let $N$ be a $S$-submodule of $M$. As before, by restriction of scalars,
$N$ is an $R$-module. Suppose that $P$ is an $R$-submodule of $M$. 
Then, $\mfa\cdot P\subseteq\mfa\cdot M=0$ and $P$ is also an $S=R/\mfa$-module.
Therefore, the set of $R$-submodules of $M$ and the set of $S$-submodules of $M$ coincide, leading to
the equality $\length_R(M)=\length_S(M)$.

\vspace*{0,2cm}

\noindent \underline{\sc Step $5$}. 

\vspace*{0,2cm}

\noindent The Frobenius number $\frob(\mathcal{T})$ of a numerical semigroup $\mathcal{T}$ is the largest integer that does not belong to $\mathcal{T}$. If $\mathcal{T}=\langle a,b\rangle$, with $\gcd(a,b)=1$, then it is well-known that $\frob(\mathcal{T})=ab-a-b$ 
(see, e.g., \cite[Theorem~2.1.1]{ramirez}).

\vspace*{0,2cm}

\noindent \underline{\sc Step $6$}. 

\vspace*{0,2cm}

\noindent Recall that $K(\mbk\lser t\rser)=
\mbk\llaur t\rlaur:=\{\sum_{n=-r}^{\infty}\lambda_nt^n\mid \lambda_n\in\mbk, r\in\mbn\}$, the
ring of formal Laurent series over $\mbk$, though we do not used it explicitly. 
Indeed, any $f\in\mbk\lser t\rser$, $f\neq 0$, with $\ord(f)=r$, 
can be written as $f=t^r\cdot g$, where $r=\ord(f)$ and
$g$ is invertible in $\mbk\lser t\rser$. Then, the inverse of $f$ in its field of fractions is $f^{-1}=t^{-r}\cdot g^{-1}\in\mbk\lser t\rser[t^{-1}]=\mbk\llaur t\rlaur$. This shows 
$K(\mbk\lser t\rser)\subseteq \mbk\llaur t\rlaur$. Conversely, if $f\in\mbk\llaur t\rlaur$, $f\neq 0$, then
$f$ can be written as $f=g/t^r$, where $r\in\mbn$ and $g\in\mbk\lser t\rser$. Thus, $f$ is the quotient of
two series and so $f\in K(\mbk\lser t\rser)$.

\vspace*{0,2cm}

\noindent \underline{\sc Step $7$}. 

\vspace*{0,2cm}

\noindent Suppose that $A\subset B$ is an integral extension of 
integral domains $A$ and $B$ whose fields of fractions coincide, i.e., $K(A)=K(B)$. Let 
$A'=\{f\in K(A)\mid f\mbox{ is integral over }A\}$ be
the integral closure of $A$ in $K(A)$. If $B$ is integrally closed in $K(B)$, then $A'=B$.
Indeed, since $A\subset B$ is an integral extension, where $B\subset K(B)=K(A)$, then $B\subseteq A'$. Conversely, if $f\in A'$, then $f$ is integral over $A$ and, by definition, $f$ is also integral over $B$. 
Since $B$ is integrally closed in $K(B)$, then $f\in B$. 

\vspace*{0,2cm}

\end{remark}

\begin{remark}\label{remark-sally1} {\sc Sally's diagram (I).}
The following diagram, with natural vertical morphisms, is considered in the book of Sally \cite[p.~ 61]{sally}. 
(The down-to-up vertical morphisms come from $\mbz_{(p)}\to\mbz_{(p)}/p\mbz_{(p)}=\mbz/p\mbz$.)
Now, let
$\tilde{\rho}(x)=t^6+t^{31}$, $\tilde{\rho}(y)=t^8$ and $\tilde{\rho}(z)=t^{10}$, 
and suppose that $\overline{\rho}$ and $\rho$ are defined in an analogous way. 
\begin{center}
\begin{tikzcd}[row sep = huge, column sep = huge]
(\mbz/p\mbz)\lser x,y,z\rser \arrow[r, "\overline{\rho}"]  
&(\mbz/p\mbz)\lser t\rser 
\\
\mbz_{(p)}\lser x,y,z\rser 
\arrow[r, "\tilde{\rho}"] \arrow[d, "", swap] \arrow[u, "", swap]                            
&\mbz_{(p)}\lser t\rser \arrow[d, "\iota"] \arrow[u, "", swap] \\
\mbq\lser x,y,z\rser \arrow[r, "\rho", swap] 
&\mbq\lser t\rser                            
\end{tikzcd}
\end{center}
Let $P=\ker(\rho)$, $\tilde{P}=\ker(\tilde{\rho})$ and $\overline{P}=\ker(\overline{\rho})$. 
Theorem~\ref{theo:char0} says that $P$ is minimally generated by the polynomials 
\begin{eqnarray*}
&&f_1=\underline{3y^3-4xyz+x^4}-3y^3z^5-2xy^6z^2-x^2y^4z^3,\phantom{+}
f_2=\underline{2y^2z-3xz^2+x^3y}-2y^7z^2-xy^5z^3,\\
&&f_3=\underline{yz^2-3x^2y^2+2x^3z}-y^6z^3-2xy^4z^4,\phantom{+}
f_4=\underline{z^3-2xy^3+x^2yz}-y^5z^4.
\end{eqnarray*}
Since the $f$'s have integer coefficients, they can also be considered in $\mbz_{(p)}\lser x,y,z\rser$ and clearly belong to $\tilde{P}$ because $\iota$ is injective. Therefore, the classes 
$\overline{f_i}$ of the $f_i$ in $(\mbz/p\mbz)\lser x,y,z\rser$ belong to $\overline{P}$.
If $p=2$, then
\begin{eqnarray*}
&&\overline{f_1}=\underline{y^3+x^4}+y^3z^5+x^2y^4z^3,\phantom{+}
\overline{f_2}=\underline{xz^2+x^3y}+xy^5z^3,\\
&&\overline{f_3}=\underline{yz^2+x^2y^2}+y^6z^3,\phantom{+}
\overline{f_4}=\underline{z^3+x^2yz}+y^5z^4.
\end{eqnarray*}
If $p=3$, then
\begin{eqnarray*}
&&\overline{f_1}=\underline{-xyz+x^4}+xy^6z^2-x^2y^4z^3,\phantom{+}
\overline{f_2}=\underline{-y^2z+x^3y}+y^7z^2-xy^5z^3,\\
&&\overline{f_3}=\underline{yz^2-x^3z}-y^6z^3+xy^4z^4,\phantom{+}
\overline{f_4}=\underline{z^3+xy^3+x^2yz}-y^5z^4.
\end{eqnarray*}
\end{remark}

\begin{discussion}\label{discussion-2}
Suppose that $\charac(\mbk)=2$. Let us find $g_1,g_2\in P$, $g_i\neq 0$, such that
$g_1^{\sigma}\in V_{10}$ and $g_2^{\sigma}\in V_{12}$.  Observe that now 
$s=\min\{\sord(f)\mid f\in P,\; f\neq 0\}=10$ and $\xi=3$. Theorem~\ref{moh43-2} says that, in order to find a minimal generating set of $P$, one can 
search for linearly independent subsets of $V_{10}$, $V_{11}$ and $V_{12}$. 
However, $W_{11}=\langle xy^2,x^2z\rangle=\langle \mon^{\alpha},\mon^{\beta}\rangle$, with $\alpha=(1,2,0)$ 
and $\beta=(2,0,1)$, and where $b_{\alpha_1,1}=b_{1,1}=1$ and $b_{\beta_1,1}=b_{2,1}=0$. Thus, by Lemma~\ref{lemmaVr0}, $V_{11}=\{0\}$.

By Lemma~\ref{lemma-s}, we can pick $g_1=\underline{z^2+x^2y}+y^5z^3$. With the notations as in 
Remark~\ref{remark-sally1}, $\overline{f_2}=xg_1$, $\overline{f_3}=yg_1$, $\overline{f_4}=zg_1$. Moreover, 
let $g_2$ be chosen as the polynomial corresponding to the element $\overline{f_1}$ in Remark~\ref{remark-sally1}.
Concretely, let $g_2=\underline{y^3+x^4}+y^3z^5+x^2y^4z^3$. Then, $g_2^{\sigma}\in V_{12}$, $g_2\in P$ and $\overline{f_1}=g_2$. 
\end{discussion}

The proof of the next theorem is very similar to the proof of Theorem~\ref{theo:char0}. Consequentely, in each step, we only give the demonstration of those details which differ from Theorem~\ref{theo:char0}. We maintain the same notations unless stated otherwise.

\begin{theorem}\label{theo:char2}
If $\charac(\mbk)=2$, then $P$ is minimally generated by 
\begin{eqnarray*}
g_1=\underline{z^2+x^2y}+y^5z^3,\phantom{+}g_2=\underline{y^3+x^4}+y^3z^5+x^2y^4z^3.
\end{eqnarray*} 
\end{theorem}
\begin{proof}
Set $I=(g_1,g_2)$. By Discussion~\ref{discussion-2}, $I\subseteq P$. 

{\sc Step 1}. By Lemma~\ref{lemma-s}, $s=10$. By
Theorem \ref{moh43-2}, since $\xi=3$, it follows that $g_1,g_2$ is part of a minimal system of generators of $P$, so $\mu(P)\geq 2$. But, if $I=P$, then $\mu(P)\leq 2$. So, $\mu(P)=2$ and $g_1,g_2$ is a minimal generating set of $P$.

{\sc Step 2}. Observe that $(g_1,g_2,y)=(z^2,x^4,y)$ and $z^2,x^4,y$ is a regular sequence. The rest follows as in Theorem~\ref{theo:char0}.

{\sc Step 3}. Note that $J=I+yR=(z^2,x^4,y)$ and $\overline{J}=J/yR=(z^2,x^4)$. 
Consider the exact sequences $0\rightarrow \mfm_S/\overline{J}\rightarrow S/\overline{J}\rightarrow 
S/\mfm_S\rightarrow 0$ and 
$0\rightarrow \mfm_S^{2}/\overline{J}\rightarrow\mfm_S/\overline{J}\rightarrow\mfm_S/\mfm_S^2\rightarrow 0$.
By the additivity of length, $\length_{S}(S/\overline{J})=3+\length_S(\mfm_S^2/\overline{J})$. Let $H$ be the ideal of $S$ generated by $q_1=x^3z$, $q_2=x^3$, $q_3=x^2z$, $q_4=x^2$ and $q_5=xz$. Let
\begin{eqnarray*}
\overline{J}_0=\overline{J}\mbox{ , }
\overline{J}_i=\overline{J}_{i-1}+(q_i), 
\text{ for } 1\leq i\leq 5.
\end{eqnarray*}
Then, $\overline{J}_5=\overline{J}+H=\mfm_S^2$. Since $\mfm_S
\overline{J}_i\subseteq \overline{J}_{i-1}$, it follows that, for $i=1,\ldots,5$,
$\overline{J}_i/\overline{J}_{i-1}=(\overline{q}_i)$ is a
one-dimensional vector space, so a simple $S$-module. For $i=1,\ldots,5$, consider the following exact sequences of $S$-modules:
\begin{eqnarray*}
0\rightarrow \overline{J}_{i-1}/\overline{J}\rightarrow
\overline{J}_{i}/\overline{J}\rightarrow\overline{J}_i/\overline{J}_{i-1}\rightarrow 0.
\end{eqnarray*}
Hence,  $\length_S(\mfm_S^2/\overline{J})=\sum_{i=1}^{5}\length_S(\overline{J}_{i}/\overline{J}_{i-1})=5$. 
Thus, $\length_S(S/yS)=8$.

{\sc Steps} 4 to 10 work as in Theorem~\ref{theo:char0}. 
\end{proof}

\begin{remark}\label{remark-sally2} {\sc Sally's diagram (II).}
Theorem~\ref{theo:char2} says that $\overline{P}$ is minimally generated by the polynomials $g_1,g_2$. 
Moreover, we have shown that, $\overline{f_1}=g_2$, 
$\overline{f_2}=xg_{1}$, $\overline{f_3}=yg_1$ and $\overline{f_4}=zg_1$. 
Note that $g_1$ is in 
\begin{eqnarray*}
(\mbz/2\mbz)\lser x,y,z\rser=(\mbz_{(2)}/(2)\mbz_{(2)})\lser x,y,z\rser
=\mbz_{(2)}\lser x,y,z\rser/2\mbz_{(2)}\lser x,y,z\rser
\end{eqnarray*}
(see, e.g., \cite[Page~5]{matsumura}).
A representative of $g_1$ in $\mbz_{(2)}\lser x,y,z\rser$ is $\varphi=z^2+x^2y+y^5z^3$.
Let us prove that  $\varphi$ is not in the ideal 
$\tilde{P}+2\mbz_{(2)}\lser x,y,z\rser$, so, in particular, 
$g\not\in (\tilde{P}+2\mbz_{(2)}\lser x,y,z\rser)/(2\mbz_{(2)}\lser x,y,z\rser)$.
Suppose that $\varphi\in \tilde{P}+2\mbz_{(2)}\lser x,y,z\rser$. Then,
$\varphi=f+2g$, where $\tilde{\rho}(f)=0$ and 
$g\in\mbz_{(2)}\lser x,y,z\rser$. Thus,
$0=-\tilde{\rho}(f)=2\tilde{\rho}(g)-\tilde{\rho}(\varphi)=
2(\tilde{\rho}(g)-(t^{20}+t^{45}+t^{70}))$. Since 
$\mbz_{(2)}\lser t\rser$ is a domain, it follows that 
$\tilde{\rho}(g)=t^{20}+t^{45}+t^{70}$. However, 
such a $g$ does not exist. Indeed, write $g=\sum_{s\geq r}\left(\sum_{\alpha\in {\rm F}(s,\mathcal{S})}\lambda_{\alpha}\mon^{\alpha}\right)$, 
$\lambda_{\alpha}\in\mbz_{(2)}$, and $\sord(g)=r$.
Although we are working with
$\tilde{\rho}:\mbz_{(2)}\lser x,y,z\rser\to \mbz_{(2)}\lser t\rser$, 
so the coefficients are not in a field $\mbk$, but in $\mbz_{(2)}$, 
some of the former results still hold, like for instance the
equality \eqref{eq-rhoh} in Lemma~\ref{lemmaVr0}. Therefore, 
\begin{multline*}\label{eq-rhotitlla}
\tilde{\rho}(g)=\sum_{s\geq r}\left[
\sum_{k=0}^{\nu_s}\left(\sum_{\alpha\in{\rm F}(s,\mathcal{S})}b_{\alpha_1,k}\lambda_{\alpha}\right)t^{2s+25k}\right]=
\\
\left(\sum_{\alpha\in{\rm F}(r,\mathcal{S})}b_{\alpha_1,0}\lambda_{\alpha}\right)t^{2r}+
\mbox{ terms in }t\mbox{ of higher degree}.
\end{multline*}
If $\tilde{\rho}(g)=t^{20}+t^{45}+t^{70}$, then, necessarily, $r=10$. 
Since $\fac(10,\nums)=\{(0,0,2),(2,1,0)\}$, then $g^{\sigma}=
\lambda z^2+\mu x^2y$, 
for some $\lambda,\mu\in\mbz_{(2)}$, and $\tilde{\rho}(g^{\sigma})=
(\lambda +\mu)t^{20}+2\mu t^{45}+\mu t^{70}$. 
Note that $45=2\cdot 10+25$ and that 
there is no $s\geq 11$ with $2s+25k=45$, because $2s$ is even, $45$ is odd, so $k$ should be non-zero and 
then $2s+25k\geq 47>45$. Therefore,  the only coefficient of $t^{45}$ in $\tilde{\rho}(g)$ comes from 
$\tilde{\rho}(g^{\sigma})$. Thus, $2\mu=1$, 
a contradiction because $\mu$ is in $\mbz_{(2)}$. This implies 
\begin{eqnarray*}
\overline{P}\not\subset (\tilde{P},2)\mbz_{(2)}\lser x,y,z\rser/2\mbz_{(2)}\lser x,y,z\rser,
\end{eqnarray*} 
showing that the equality in \cite[page 62, line 1]{sally} is not correct.

Observe that when passing to characteristic $2$ it emerges the element $g_1\in\overline{P}$, whose
$\sigma$-order is smaller than the $\sigma$-order of $\overline{f_1}$, $\overline{f_2}$, $\overline{f_3}$, $\overline{f_4}$. The equalities $\overline{f_2}=xg_{1}$, $\overline{f_3}=yg_1$ and $\overline{f_4}=zg_1$, reveal that $f_2$, $f_3$ and $f_4$ can be substituted by $g_1$. ``This'' is the reason why the minimal number of generators of $P$ might decrease when passing to characteristic $p$. 
\end{remark}

\begin{discussion}\label{discussion-3} 
Suppose that $\charac(\mbk)=3$. Let us find $h_1,h_2,h_3\in P$, $h_i\neq 0$, such that
$h_1^{\sigma}\in V_{9}$, $h_2^{\sigma}\in V_{15}$ and $h_3^{\sigma}\in V_{16}$. 
Now $s=\min\{\sord(f)\mid f\in P,\; f\neq 0\}=9$ and $\xi=3$. Theorem~\ref{moh43-2} suggests 
to find elements whose $\sigma$-leading form are in $V_{9}$, $V_{10}$ and $V_{11}$. However, 
$W_{10}=\langle z^2,x^2y\rangle=\langle\mon^{\alpha},\mon^{\beta}\rangle$, with 
$\alpha=(0,0,2)$ and $\beta=(2,1,0)$, and where $b_{\alpha_1,1}=0$ and $b_{\beta_1,1}=2$.
By Lemma~\ref{lemmaVr0}, $V_{10}=\{0\}$. Similarly, 
$W_{11}=\langle xy^2,x^2z\rangle=\langle\mon^{\alpha},\mon^{\beta}\rangle$, with $\alpha=(1,2,0)$ and 
$\beta=(2,0,1)$, and where $b_{\alpha_1,1}=1$ and $b_{\beta_1,1}=2$. 
By Lemma~\ref{lemmaVr0}, $V_{11}=\{0\}$. In other words, the only element in a minimal generating set 
of $P$ that we can deduce from Theorem~\ref{moh43-2} is $h_1$. In Remark~\ref{remark-v12v13} 
we see that $V_{12}$, $V_{13}$ are non-zero, but they do not lead to a minimal generating set of $P$.

By Lemma~\ref{lemma-s}, we can pick $h_1=\underline{yz-x^3}-y^6z^2+xy^4z^3$. With the notations as in 
Remark~\ref{remark-sally1}, $\overline{f_1}=-xh_1$, $\overline{f_2}=-yh_1$, $\overline{f_3}=zh_1$.

Take $h_2^{\sigma}=az^3+bxy^3+cx^2yz\in W_{15}$, omitting again the $x^5$ for the sake of simplicity. Then,
\begin{eqnarray*}
\rho(h_2^{\sigma})=(a+b+c)t^{30}+(b+2c)t^{55}+ct^{80}.    
\end{eqnarray*}
Choose $(a,b,c)=(1,1,1)$, so $h_2^{\sigma}=z^3+xy^3+x^2yz$ and $\rho(h_2^{\sigma})=t^{80}$. So we can pick
$h_2^{\tau}=-y^5z^4$. Therefore, let $h_2=\underline{z^3+xy^3+x^2yz}-y^5z^4$.

Take $h_3^{\sigma}=ay^4+bxy^2z+cx^2z^2\in W_{16}$, omitting the $x^4y$ for the sake of simplicity. Then, 
\begin{eqnarray*}
\rho(h_3^{\sigma})=(a+b+c)t^{32}+(b+2c)t^{57}+ct^{82}.    
\end{eqnarray*}
Choose $(a,b,c)=(1,1,1)$, so $h_3^{\sigma}=y^4+xy^2z+x^2z^2$ and $\rho(h_3^{\sigma})=t^{82}$. So we can pick
$h_3^{\tau}=-y^4z^5$. Therefore, let $h_3=\underline{y^4+xy^2z+x^2z^2}-y^4z^5$.
\end{discussion}

\begin{theorem}\label{theo:char3} 
If $\charac(\mbk)=3$, then $P$ is minimally generated by 
\begin{eqnarray*}
&&h_1=\underline{yz-x^3}-y^6z^2+xy^4z^3,\phantom{+}
h_2=\underline{z^3+xy^3+x^2yz}-y^5z^4,\phantom{+}
h_3=\underline{y^4+xy^2z+x^2z^2}-y^4z^5.  
\end{eqnarray*} 
\end{theorem}
\begin{proof}
Set $I=(h_1,h_2,h_3)$. By Discussion~\ref{discussion-3}, $I\subseteq P$. 

Here {\sc Step 1} is more laborious than in Theorem~\ref{theo:char2}. By Lemma~\ref{lemma-s}, $s=9$. 
By Theorem~\ref{moh43-2}, since $\xi=3$, it follows that $h_1$ is part of a minimal system of generators of $P$ and $\mu(P)\geq 1$. Since $h_1\in (x,y)$ and $h_2=z^3+f$, with $f\in (x,y)$, it follows that $h_2\not\in (h_1)$. Similarly, since $h_1\in (x,z)$ and $h_3=y^4+g$, with $g\in (x,z)$, it follows that $h_3\not\in (h_1)$. If $\mu(P)=1$, then $P=(h_1)$ and, since $h_2,h_3\in P$, then $h_2,h_3\in (h_1)$, a contradiction. Therefore, $\mu(P)>1$.
Suppose that $\mu(P)=2$. Then, there exist $f\in R$ such that $h_1,f$ is a minimal system of generators of $P$. Since $h_2,h_3\in P=(h_1,f)$, it follows that $h_2=a_1h_1+a_2f$, $h_3=b_1h_1+b_2f$, where $a_i,b_i\in R$, 
$1\leq i\leq 2$. Since $h_2,h_3\not\in(h_1)$, it follows that $a_2\neq0$ and $b_2\neq 0$. Moreover, if $I=P$, then $f=r_1h_1+r_2h_2+r_3h_3$, for some $r_i\in R$, $1\leq i\leq 3$. Clearly, $(r_2,r_3)\neq (0,0)$, otherwise $f\in(h_1)$, a contradiction. Substituting the expressions of $h_2$ and $h_3$ in the equality $f=r_1h_1+r_2h_2+r_3h_3$, we obtain $(1-r_2a_2-r_3b_2)f=(r_1+r_2a_1+r_3b_1)h_1$. If $1-r_2a_2-r_3b_2$ is a unit, then $f\in(h_1)$, a contradiction. Hence, $1-r_2a_2-r_3b_2$ is not a unit, i.e., 
 $1-r_2a_2-r_3b_2\in\mfm$. Therefore, either $r_2a_2\not\in\mfm$, or else $r_3b_2\not\in\mfm$. 
 In the first case, $a_2$ is a unit, $f\in (h_1,h_2)$, $P=(h_1,f)\subseteq (h_1,h_2)\subseteq I=P$, 
 so $h_3\in I=(h_1,h_2)\subset (x,z)$, which is a contradiction, because $h_3\not\in (x,z)$. In the second case, 
 $b_2$ is a unit, $f\in (h_1,h_3)$, $P=(h_1,f)\subseteq (h_1,h_3)\subseteq I=P$, so $h_2\in I=(h_1,h_3)\subset (x,y)$, which is a contradiction,
 because $h_2\not\in (x,y)$. Thus, $\mu(P)>2$. But, if $I=P$, then $\mu(P)\leq 3$. 
 So, $\mu(P)=3$ and $h_1,h_2,h_3$ is a minimal generating set of $P$.

{\sc Step 2}. 
Now $(h_1,h_2,y)=(x^3,z^3,y)$ and $x^3,z^3,y$ is a regular sequence. Then, proceed as in Theorem~\ref{theo:char0}.

{\sc Step 3}. Observe that $J=I+yR=(x^3,z^3,x^2z^2,y)$ and $\overline{J}=J/yR=(x^3,x^2z^2,z^3)$. 
Note that $\mfm_S^4\subset \overline{J}\subset\mfm_S^{3}$, where $S=R/yR$ and $\mfm_S=\mfm/yR$. 
For $0\leq i\leq 2$, and understanding $\mfm_S^0=S$, consider the short exact sequences:
\begin{eqnarray*}
0\rightarrow \mfm_S^{i+1}/\overline{J}\rightarrow\mfm_S^{i}/\overline{J}\rightarrow\mfm_S^{i}/\mfm_S^{i+1}\rightarrow 0.
\end{eqnarray*}
By the additivity of the length, $\length_{S}(S/\overline{J})=6+\length_S(\mfm_S^3/\overline{J})$. Since $\mfm_S^4\subset \overline{J}\subset\mfm_S^{3}$, it follows that $\mfm_S^{3}/\overline{J}=\langle x^2z,xz^2\rangle$ is a two-generated $S/\mfm_S$-vectorial space. Thus, $\length_S(S/\overline{J})=8$.

The rest is an in Theorem~\ref{theo:char0}. 
\end{proof}

\begin{remark}\label{remark-v12v13}
In Discussion~\ref{discussion-3}, we could have looked for elements $\hat{h}_2$ and $\hat{h}_3$, such that $\hat{h}_2^{\sigma}\in V_{12}$ and $\hat{h}_3^{\sigma}\in V_{13}$. By Lemma~\ref{lemmaVr0}, 
we could have taken, for instance, 
\begin{eqnarray*}
\hat{h}_2=\underline{-xyz+x^4}-xy^6z^2+y^3z^5-y^3z^{10}\phantom{+}\mbox{ and }\phantom{+}
\hat{h}_3=\underline{-y^2z+x^3y}+y^2z^6-xz^7.
\end{eqnarray*}
However, $(h_1,\hat{h}_2,\hat{h}_3)\subsetneq P=(h_1,h_2,h_3)$. Indeed, $h_2$ can be written as $h_2=z^3+g$, where $g\in (x,y)$ and $(h_1,\hat{h}_2,\hat{h}_3)\subset (x,y)$, thus $h_2\not\in (h_1,\hat{h}_2,\hat{h}_3)$, otherwise $z^3\in (x,y)$.
\end{remark}

\begin{remark}\label{remark-sally3} {\sc Sally's diagram (III).}
Theorem~\ref{theo:char3} says that $\overline{P}$ is minimally generated by the polynomials $h_1,h_2,h_3$. 
Moreover, we have seen that $\overline{f_1}=-xh_1$, $\overline{f_2}=-yh_1$, $\overline{f_3}=zh_1$ and $\overline{f_4}=h_2$. Thus, $\overline{f_1}$, $\overline{f_2}$ and $\overline{f_3}$ can be substituted by $h_1$. 
Nevertheless, $h_1$ and $h_2$ are not sufficient to generate $P$. We need an extra $h_3$, which is not a divisor of any of the $\overline{f_i}$.
\end{remark}

To finish the section we briefly treat the prime ideal $P_1$ of Moh.  
Recall that the prime ideals $P_n$ of Moh are defined as follows:
$n\geq 1$ is an odd integer, $m=(n+1)/2$, and $\lambda$ is an integer greater than $n(n+1)m$, with $\gcd(\lambda,m)=1$; $\rho_n:\mbk\lser x,y,z\rser\to\mbk\lser t\rser$ is the $\mbk$-algebra morphism with $\rho_n(x)=t^{nm}+t^{nm+\lambda}$, $\rho_n(y)=t^{(n+1)m}$ and $\rho_n(z)=t^{(n+2)m}$ and $P_n:=\ker(\rho_n)$. 
Thus, if $n=1$, then $m=1$, $\lambda$ is any integer with $\lambda\geq 3$, and $\rho(x)=t(1+t^{\lambda})$, $\rho(y)=t^2$ and $\rho(z)=t^3$. For an alternative approach to $P_1$, in characteristic zero, see \cite[Theorem~3.1]{mss}. Our two generators $f_1,f_2$, for the prime ideal $P_1$, which are valid in any characteristic, have been found with the help of {\sc Singular} \cite{dgps} and do not coincide, in general, with the ones in \cite{mss}.

\begin{proposition}\label{prop-casen1}
If $\lambda=2r+1$, with $r\geq 1$, then $P_1=(y-x^2+y^{r}z+xy^{r+1},z-xy+y^{r+2})$. If $\lambda=2r$, with $r\geq 2$, then $P_1=(y-x^2+y^{r+1}+xy^{r-1}z,z-xy+y^rz)$. In particular, $\mu(P_1)=2$, independently of the characteristic of the field $\mbk$.
\end{proposition}
\begin{proof}
As always, set $R=\mbk\lser x,y,z\rser$.
Let $f_1:=y-x^2+y^{r}z+xy^{r+1}$ and $f_2:=z-xy+y^{r+2}$, if $\lambda$ is odd; let
$f_1:=y-x^2+y^{r+1}+xy^{r-1}z$ and $f_2:=z-xy+y^rz$, if $\lambda$ is even. Set $I=(f_1,f_2)$. One can check that $I\subseteq P_1$. Since $(f_1,f_2,y)=(x^2,y,z)$ and $x^2,y,z$ is a regular sequence, it follows that $f_1,f_2$ is a regular sequence and that $\height(I)=2$. 
So, it suffices to prove that $I$ is a prime ideal. 
Let $\varrho:\mbk\lser x,y,z\rser\to\mbk\lser x,y\rser$ be the $\mbk$-algebra morphism defined as $\varrho(x)=x$, $\varrho(y)=y$ and $\varrho(z)=xy-y^{r+2}$, if $\lambda=2r+1$, or $\varrho(z)=xy(1+y^r)^{-1}$, if $\lambda=2r$. 
Clearly, $(f_2)\subseteq \ker(\varrho)$ and $f_2$ is irreducible. Since $R$ is factorial, $(f_2)$ is a prime ideal of height 1 contained in the height-one prime ideal $\ker(\varrho)$. Therefore, $(f_2)=\ker(\varrho)$ and
$\varrho$ induces a $\mbk$-algebra isomorphism $R/f_2R\cong \mbk\lser x,y\rser$. In particular, 
$R/I\cong\mbk\lser x,y\rser/(y-x^2+y^2g(x,y))$, for some $g(x,y)\in\mbk\lser x,y\rser$. Since
$y-x^2+y^2g(x,y)$ is irreducible and $\mbk\lser x,y\rser$ is factorial,
then $R/I$ is a domain. Hence, $I$ is prime. 
\end{proof}

\begin{remark}
A run with {\sc Singular} \cite{dgps} shows that, for $n=5$ and $\lambda=n(n+1)m+1$, then: 
$\mu(P_5)=6$ if $\charac(\mbk)=0$ or $p\geq 7$, $\mu(P_5)=3$ if $\charac(\mbk)=2$, $\mu(P_5)=2$ if $\charac(\mbk)=3$, and $\mu(P_5)=5$ if $\charac(\mbk)=5$. To prove it, one could reproduce our proof given for the case $n=3$. For $n=7$ or higher, it seems that the computation with {\sc Singular} stops and does not provide an answer, at least in characteristic zero.
\end{remark}

\section{Standard Bases}\label{section-standard}

The purpose of this section is to prove that the minimal generating sets of the prime ideal $P_3$ of Moh constructed in Theorems~\ref{theo:char0},~\ref{theo:char2} and ~\ref{theo:char3} are standard bases with respect to the negative degree reverse lexicographic order $>_{\rm ds}$. Concretely, we display the Mora standard representation for the $S$-polynomials (see \cite[Section~3]{mora1}, see also \cite[Definition~2.3.8 and Theorem~4.4.16]{ds} and \cite[Definition~1.6.9 and Algorithm~1.7.6]{gp}). We start by recalling some definitions and results. We mainly follow 
 \cite[Sections~1.5 to 1.7 and 6.4]{gp} and \cite[Section~4.4]{ds}.

\begin{notation}\label{notation-ord}
A {\em monomial ordering} $>$ is a total ordering on $\monset_d=\{\mon^{\alpha}\mid \alpha\in\mbn^d\}$ satisfying that if $\mon^\alpha>\mon^{\beta}$, then $\mon^{\gamma}\mon^{\alpha}>\mon^{\gamma}\mon^{\beta}$, for all $\alpha,\beta,\gamma\in\mbn^{d}$. 
A monomial ordering $>$ is said to be {\em local} if $\mon^\alpha<1$, for all $\alpha\neq(0,\ldots,0)$. A local monomial ordering $>$ is a {\em local degree ordering} if $\deg(\mon^\alpha)<\deg(\mon^\beta)$ implies that $\mon^\alpha>\mon^\beta$. A non-zero polynomial $f$ can be written $f=\sum_{\nu=0}^{n} a_{\nu}\mon^{\alpha(\nu)}$, $a_\nu\in\mbk\setminus\{0\}$, where $\mon^{\alpha(0)}>\ldots>\mon^{\alpha(n)}$. The {\em leading term} of $f$ is $\lterm(f)=a_{0}\mon^{\alpha(0)}$ and the {\em leading monomial} of $f$ is 
$\lmon(f)=\mon^{\alpha(0)}$. 

Let $>$ be a monomial ordering. Let $\mbk[\monx]_{>}:=\{f/u\mid f\in\mbk[\monx],u\in S_{>}\}$ be 
the localization of $\mbk[\monx]$ with respect to the multiplicative closed set 
$S_{>}:=\{u\in\mbk[\monx]\setminus\{0\}\mid\lmon(u)=1\}$. If $>$ is local, then $\mbk[\monx]_{>}=\mbk[\monx]_{(\underline{\rm x})}$. For $f\in \mbk[\monx]_>$, choose $u\in\mbk[\monx]$ such that $\lterm(u)=1$ and $uf\in\mbk[\monx]$. Then $\lterm(f):=\lterm(uf)$ and $\lmon(f):=\lmon(uf)$. 

Given a subset $G\subset\mbk[\monx]_>$, the {\em leading ideal} of $G$ is defined as $L(G)=\langle \lmon(g) \mid g\in G, g\neq 0\rangle_{\Bbbk[\underline{\rm x}]}$, which is seen as an ideal of $\mbk[\monx]$. Let $I$ be an ideal of $\mbk[\monx]_>$. A finite subset set $G\subset\mbk[\monx]_>$ is a {\em standard basis} of $I$ if $G\subset I$ and $L(I)=L(G)$. In the non-local case, a standard basis is also called a Gr\"obner basis.
\end{notation}

\begin{definition}
Let $>$ be a  monomial ordering. Let $\mathcal{G}$ be the set of all finite lists in $\mbk[\monx]_{>}$. A {\em weak normal form} on $\mbk[\monx]_>$ is a map $\nf:\mbk[\monx]_>\times\mathcal{G}\to\mbk[\monx]_>$ satisfying: 
\begin{itemize}
\item[$(0)$] $\nf(0\mid G)=0$ for all $G\in\mathcal{G}$;
\item[$(1)$] If $\nf(f\mid G)\neq 0$, then $\lmon(\nf(f\mid G))\not\in L(G)$, for all $f\in\mbk[\monx]_>$ and all $G\in\mathcal{G}$;
\item[$(2)$] For all $f\in\mbk[\monx]_>$ and all $G=g_1,\ldots, g_m$, $G\in\mathcal{G}$, there exist $u\in(\mbk[\monx]_>)^*$ such that $uf$ has a {\em standard representation} with respect to $\nf(-\mid G)$, that is, $uf-\nf(uf\mid G)=\sum_{i=1}^{m}a_ig_i$, with $a_i\in 
\mbk[\monx]_{(\underline{\rm x})}$ satisfying that 
$\lmon(\sum_{i=1}^{m}a_ig_i)\geq\lmon(a_ig_i)$, for all $i$ such that $a_ig_i\neq 0$.
\end{itemize}
A weak normal form is called {\em polynomial} if whenever $f\in\mbk[\monx]$ and $G$ is a list in $\mbk[\monx]$, there exists a $u\in(\mbk[\monx]_>)^*\cap\mbk[\monx]$ such that $uf$ has a standard representation with $a_i\in\mbk[\monx]$, for all $1\leq i\leq m$. 

Fixed a list $G$, a {\em weak normal form with respect to $G$} is a map $\nf:\mbk[\monx]_{>}\to\mbk[\monx]_{>}$, $f\mapsto \nf(f\mid G)$, satisfying $(0)$, $(1)$ and $(2)$, where $G$ is fixed. Similarly, on can define a {\em polynomial weak normal form with respect to $G$}.
\end{definition}

\begin{definition}
Let $>$ be any monomial ordering. Let $f\in\mbk[\monx]$ and let $G=g_1,\ldots,g_m$ 
be a finite list of $\mbk[\monx]$. 
The {\em Mora normal form algorithm} returns polynomials $u,a_i,h\in\mbk[\monx]$, such that $\lmon(u)=1$,
$uf-h=\sum_{i=1}^{m}a_ig_i$, where $\lmon(f)\geq\lmon(a_ig_i)$ for all $a_ig_i\neq 0$ and, if $h\neq 0$, then $\lmon(h)$ is not divisible by any $\lmon(g_i)$ (see, e.g., \cite[Theorem~4.4.16]{ds} and \cite[Algorithm~1.7.6]{gp}). The reminder $h$ of this expression is called the {\em Mora Normal Form} of $f$ with respect to $G$ and will be denote it as $\nfmora(f\mid G)$. 
\end{definition}

\begin{remark}
In fact, $\nfmora:\mbk[\monx]_{>}\to\mbk[\monx]_{>}$ is a polynomial weak normal form on $\mbk[\monx]_{>}$ with respect to a fixed list $G$ in $\subset\mbk[\monx]$. Indeed, given $f\in\mbk[\monx]_{>}$, write
$f=g/u$, where $g,u\in\mbk[\monx]$, $u\neq 0$, $\lmon(u)=1$, which is unique up to scalars. Then, one defines
$\nfmora(f\mid G)=\nfmora(g\mid G)/u$. 
\end{remark}

\begin{remark}
Let $>$ be a local degree ordering on $\monset_d$. A non-zero element $f\in\mbk\lser\monx\rser$ can be written as $f=\sum_{\nu\geq 0}^{\infty}a_{\nu}\mon^{\alpha(\nu)}$, $a_\nu\in\mbk\setminus\{0\}$, where $\mon^{\alpha(\nu)}>\mon^{\alpha(\nu+1)}$. The {\em leading term} of $f$ is defined as $\lterm(f)=a_{0}\mon^{\alpha(0)}$ and the {\em leading monomial} of $f$ as $\lmon(f)=\mon^{\alpha(0)}$. Similarly as before, given a subset $G\subset\mbk\lser\monx\rser$, the {\em leading ideal} of $G$ is defined as $L(G)=\langle G\rangle_{\mbk\lser\underline{\rm x}\rser}$, seen as an ideal in $\mbk\lser\monx\rser$. Furthermore, $G$ is a standard basis of an ideal $I\subset\mbk\lser\monx\rser$ if $G\subset I$ and $L(I)=L(G)$. 
\end{remark}

\begin{proposition}\label{prop:groeb}
Let $F=\{f_1,f_2,f_3,f_4\}$ be defined as in Theorem~\ref{theo:char0}. If $\charac(\mbk)=0$ or if $\charac(\mbk)\geq 5$, then $F$ is a standard basis of $P_3$ with respect to the negative degree reverse lexicographic order $>_{\rm ds}$. 
\end{proposition}
\begin{proof}
Let $J$ be the ideal of $\mbk[\monx]$ generated by $F=\{f_1,f_2,f_3,f_4\}$. By \cite[Theorem 6.4.3]{gp}, if $F$ is a standard basis of $J$, then $F$ is a standard basis of $J\mbk\lser\monx\rser=P_3$. Hence, it suffices to prove that $F$ is a standard basis of $J$. So we need to
check that $\nfmora(S(f_i,f_j)\mid F)=0$ for $1\leq i,j\leq 4$ (see \cite[Theorem 1.7.3 (Buchberger’s criterion)]{gp}). Using the Mora normal form algorithm (see Remark below), we prove that: 
\begin{eqnarray*}
u\kern-0,8em&S(f_1,f_2)=&\kern-0,8em\left(-z^6+y^5z^2+(4/3)xy^3z^3+(16/9)x^2yz^4-(1/3)x^4y^2z^2+(2/9)x^5z^3\right)f_1\\
&&\kern-1,4em-(22/9)x^3z^4f_2+\left((1/2)x-(9/2)xz^5-(1/3)x^2y^3z^2-(11/18)x^3yz^3\right)f_3+(5/3)x^4z^3f_4;\\
&S(f_1,f_3)=&\kern-0,8em3x^2yf_1-3x^3f_2-(9/2)xzf_3+zS(f_1,f_2);\\
&S(f_1,f_4)=&\kern-0,8em2xy^3f_1-x^3zf_2-\left(2xz^2+(2/3)x^3y\right)f_3+x^4f_4+\left((4/9)z^2+(4/3)x^2y\right)S(f_1,f_2)\\&&\kern-0,8em+(5/9)zS(f_1,f_3);\\
&S(f_2,f_3)=&\kern-0,8em-3xf_4; \\
&S(f_2,f_4)=&\kern-0,8em\left((4/3)xy^2-(8/9)x^2z\right)f_1-3xzf_4+(8/9)x^2S(f_1,f_2);\\   
&S(f_3,f_4)=&\kern-0,8em(2/3)xyf_1-(2/3)x^2f_2;
\end{eqnarray*}
where $u=1-z^5-(2/3)xy^3z^2-(11/9)x^2yz^3+(2/9)x^5z^2$. It follows that $\nfmora(S(f_1,f_2)\mid F)=\nfmora(S(f_2,f_3)\mid F)=\nfmora(S(f_3,f_4)\mid F)=0$. By multiplying the second equality by $u$ and 
substituting the value of $uS(f_1,f_2)$ from the first equality, we get $\nfmora(S(f_1,f_3)\mid F)=0$. Similarly, 
one can check that $\nfmora(S(f_1,f_4)\mid F)=\nfmora(S(f_2,f_4)\mid F)=0$.
\end{proof}
\begin{remark}
To compute the explicit Mora standard representation for the $S$-polynomials, we have replicated the Mora’s Division Algorithm with Singular \cite{dgps}. To our knowledge, this algorithm was originally stated in \cite[Section 3]{mora1}. We follow \cite[Algorithm 1.7.6]{gp}. We use the library \verb+teachstd.lib+ which provides the computations for the $S$-polynomials and the minimal \verb+ecart+. The input of the algorithm is a polynomial \verb|f| and a finite list of polynomials \verb|I|, and the output is the Mora normal form of \verb|f| with respect to \verb|I| denoted by \verb|h|. In each iteration, we store the divisors in the list \verb+D+ and the quotients in the list \verb+Q+. Since the Mora's Division Algorithm adds a new element in \verb|I| in certain iterations, we keep track of it in the liest \verb+L+. Let us see a more detailed explanation on how to use the code to recover the Mora standard representation for the $S(f_1,f_2)$. 
\begin{itemize}
\item Define the ring $\mbk[x,y,z]_{>}$ where $>$ is the $\ds$ ordering. 
\begin{verbatim}
LIB "teachstd.lib";
ring R=0,(x,y,z),ds; 
\end{verbatim}
\item Define the subset of polynomials $F=\{f_1,f_2,f_3,f_4\}$ as an ideal. 
\begin{verbatim}
poly f1=3y3-4xyz+x4-3y3z5-2xy6z2-x2y4z3;
poly f2=2y2z-3xz2+x3y-2y7z2-xy5z3;
poly f3=yz2-3x2y2+2x3z-y6z3-2xy4z4;
poly f4=z3-2xy3+x2yz-y5z4;
ideal F=f1,f2,f3,f4;
\end{verbatim}
\item Define the $S$-polynomial $S(f_1,f_2)$.
\begin{verbatim}
poly s12=spoly(f1,f2);
\end{verbatim}
\item Define the empty lists \verb|D|, \verb|Q| and \verb|L|, initialize \verb|h| and run the algorithm code. 
\begin{verbatim}
list D;
list Q;
list L;
poly h=s12;
while(h!=0 && minEcart(F,h)!=0){
    poly d=minEcart(F,h);
    poly q=lead(h)/lead(minEcart(F,h));
    D=insert(D,d);
    Q=insert(Q,q);
    if(ecart(d)>ecart(h)){F=F+h;L=insert(L,h);}
        else{L=insert(L,0);}
    h=h-q*d;
};
\end{verbatim}
\item Run \verb|h;| to print the Mora normal form for $S(f_1,f_2)$ with respect to $F$ and check that is indeed zero.
\item Run any of the three lists to see how many steps the algorithm needed. In this case, it needed $15$ steps. The three lists are ordered such that the first element of the list is the last one to be inserted, and so on. 
\item Do the following iterative process step by step from $k=1$ to $k=15$. 
\\
STEP $k$
\begin{itemize}
\item Let $d_k=\verb|D[n-k+1]|$ and $q_k=\verb|Q[n-k+1]|$. 
\item Let $h_k=S(f_1,f_2)-\sum_{i=1}^{k}q_id_i$. 
\item If \verb|L[n-k+1]!=0|, then the algorithm has added $h_k$ as a new element of \verb|I|. 
\item Identify $d_{k}$ as an element of \verb|I|. Note that in previous steps, we could have added some new  elements in \verb|I|. Then, $d_k=f_j$ for some $1\leq j\leq 5$ or $d_k=h_i$ for some $1\leq i<k$.
\end{itemize}
\item In the last step, we get $h_n=S(f_1,f_2)-\sum_{i=1}^{n}q_id_i$ which is the Mora normal form of $S(f_1,f_2)$ and it is zero. Hence, $S(f_1,f_2)=\sum_{i=1}^{n}q_id_i$. Note that the elements $d_i$ are either $f_i$ for some $1\leq i\leq 4$ or $h_k$ for some $1\leq k\leq n$. Moreover, $h_1\in (f_1,f_2,f_3,f_4,S(f_1,f_2))$, $h_2$ 
is in $(f_1,f_2,f_3,f_4,s_{1,2},h_1)=(f_1,f_2,f_3,f_4,S(f_1,f_2))$. Recursively, $h_k\in (f_1,f_2,f_3,f_4,S(f_1,f_2))$. This means that $d_i\in (f_1,f_2,f_3,f_4,S(f_1,f_2))$ and so we can recover the Mora standard representation of $S(f_1,f_2)$ with respect to $F$. This is $uS(f_1,f_2)=\sum_{i=1}^{s}a_if_i$ where $u,a_i\in\mbk[\monx]$ for $1\leq i\leq 4$ and $\lmon(u)=1$.
\end{itemize}
After computing the Mora standard representation for $S(f_1,f_2)$, we repeat the process for $S(f_1,f_3)$, and so on, keeping the extended ideal \verb|F| from the previous step. This provides a more efficient way of computing the Mora standard representation for all the $S$-polynomials.
\end{remark}

\begin{proposition}
Let $G=\{g_1,g_2\}$ be defined as in Theorem~\ref{theo:char2}. If $\charac(\mbk)=2$, 
then $G$ is a standard basis of $P_3$ with respect to the negative degree reverse lexicographic order $>_{\rm ds}$. 
\end{proposition}
\begin{proof}
Using the Mora division algorithm we obtain: 
\begin{eqnarray*}
(1+y^5z+x^2yz^3)S(g_1,g_2)=(x^4+y^3z^5+y^8z)g_1+(x^2y+x^2yz^5)g_2. 
\end{eqnarray*}
Therefore, $\nfmora(S(g_1,g_2)\mid G)=0$.  
\end{proof}

\begin{proposition}
Let $H=\{h_1,h_2,h_3\}$ be defined as in Theorem~\ref{theo:char3}. If $\charac(\mbk)=3$, then $H$ is a standard basis of $P_3$ with respect to the negative degree reverse lexicographic order $>_{\rm ds}$. 
\end{proposition}
\begin{proof}
Using the Mora division algorithm we obtain: 
\begin{eqnarray*}
&S(h_1,h_2)&=-xh_3;\\
\left(1-z^5\right)\kern-0,8em&S(h_1,h_3)&=\left(-xyz+xyz^6\right)h_1+\left(-x^2-xy^2z^4\right)h_2+\left(z^6-y^5z^2+xy^3z^3+x^2yz^4\right)h_3;\\
&S(h_2,h_3)&=xyz^3h_1-xy^2zh_2+\left(xy^3+x^2yz\right)h_3+z^2S(h_1,h_3).
\end{eqnarray*}
So, $\nfmora(S(h_1,h_2)\mid H)=\nfmora(S(h_1,h_3)\mid H)=0$. Multiplying the third equality by $v=1-z^5$ and substituting the value of $vS(h_1,h_3)$, we get $\nfmora(S(h_2,h_3)\mid H)=0$. The rest follows as in Proposition~\ref{prop:groeb}.
\end{proof}

\section*{Acknowledgement}

It is a pleasure to thank the comments of the referee. 


{\small
\begin{thebibliography}{cc}
\bibitem{am}{M.F. Atiyah, I.G. Macdonald, Introduction to commutative algebra. 
Addison-Wesley Publishing Co., Reading, Mass.-London-Don Mills, Ont., 1969}
\bibitem{bh}{W. Bruns, J. Herzog, Cohen-Macaulay rings. 
Cambridge Studies in Advanced Mathematics, {\bf 39}. 
Cambridge University Press, Cambridge, 1993}
\bibitem{dgps}{W. Decker, G.-M. Greuel, G. Pfister, H. Sch{\"o}nemann, 
{\sc Singular} {4-3-0} - {A} computer algebra system for polynomial computations.
\newblock {https://www.singular.uni-kl.de} (2022).}
\bibitem{ds}{W. Decker and F. Schreyer, Varieties, groebner bases, and algebraic curves. Manuscript. 
Available at: {https://www.math.uni-sb.de/ag/schreyer/index.php/teaching/ss-18/128-algebraische-geometrie}}
\bibitem{gp}{G.-M. Greuel, G. Pfister, A Singular introduction to commutative algebra. Second, extended edition. With contributions by Olaf Bachmann, Christoph Lossen and Hans Schönemann. Springer-Verlag, Berlin, 2008.}
\bibitem{matsumura}{H. Matsumura, Commutative ring theory. Translated from the Japanese by M. Reid. Second edition. Cambridge Studies in Advanced Mathematics, {\bf 8}. Cambridge University Press, Cambridge, 1989}
\bibitem{mss}{R. Mehta, J. Saha, I. Sengupta, Moh's example of algebroid space curves. J. Symbolic Comput. {\bf 104} (2021), 168–182.}
\bibitem{moh1}{T.T. Moh, On the unboundedness of generators of prime
  ideals in power series rings of three variables. J. Math. Soc. Japan
  {\bf 26} (1974), 722-734}
\bibitem{moh2}{T.T. Moh, On generators of ideals. Proc. Amer. Math. Soc. {\bf 77} (1979), no. 3, 309-312.}
\bibitem{mora1}{F. Mora, An algorithm to compute the equations of tangent cones. Computer algebra (Marseille, 1982), 158–165, Lecture Notes in Comput. Sci. {\bf 144}, Springer, Berlin-New York, 1982.}
\bibitem{ramirez}{J.L. Ram\'{\i}rez-Alfons\'{\i}n. The Diophantine
  Frobenius problem.  Oxford Lecture Series in Mathematics and its
  Applications, 30.  Oxford University Press, Oxford, 2005.}
\bibitem{sally}{J. Sally, Numbers of generators of ideals in local
  rings. Marcel Dekker, Inc., New York-Basel, 1978.}
\end{thebibliography}
}
\end{document}